\documentclass[reqno]{amsart} 
\usepackage{amsmath,amssymb,accents,amsthm} 
\usepackage[latin1]{inputenc}
\usepackage[all]{xypic}
\allowdisplaybreaks[4]
\newtheorem{theorem}{Theorem}[section]
\newtheorem{proposition}[theorem]{Proposition}
\newtheorem{claim}[theorem]{Claim}
\newtheorem{definition}[theorem]{Definition}
\newtheorem{corollary}[theorem]{Corollary} 
\newtheorem{lemma}[theorem]{Lemma} 
\theoremstyle{definition} 
\newtheorem{example}[theorem]{Example}
\numberwithin{equation}{section} 
\newcommand\bm{\begin{pmatrix}} 
\renewcommand\em{\end{pmatrix}}

\renewcommand\|{\Big|}
\renewcommand\({\Big(} 
\renewcommand\){\Big)} 
\renewcommand\[{\Big[} 
\renewcommand\]{\Big]}
\newcommand\er{\eqref} 
\newcommand\be[1]{\begin{equation}\label{#1}}
\newcommand\ee{\end{equation}} 
\newcommand\I{\int\limits}
\renewcommand\S{\sum\limits}
\renewcommand\P{\prod\limits}
\newcommand\U{\bigcup\limits}
\newcommand\bb[4]{\begin{pmatrix}{#1}&{#2}\\{#3}&{#4}\end{pmatrix}}
\renewcommand\c[1]{{\check{#1}}} 
\renewcommand\d[1]{{\dot{#1}}}
\newcommand\h[1]{{\hat{#1}}} 
\newcommand\y[1]{{\undertilde{#1}}} 
\renewcommand\t[1]{{\tilde{#1}}} 
\newcommand\w[1]{{\widetilde{#1}}} 
\renewcommand\o[1]{{\overline{#1}}}

\newcommand\F[1]{{\sqrt{#1}}} 
\newcommand\f[2]{{\frac{#1}{#2}}}
\renewcommand\b[2]{{\binom{#1}{#2}}}
\renewcommand\1{\prime}
\newcommand\dl{\partial}

\newcommand\el{\ell}
\newcommand\la{\alpha}
\newcommand\lb{\beta}

\renewcommand\lg{\gamma}
\newcommand\lG{\Gamma}
\newcommand\ld{\delta}
\newcommand\lD{\Delta}
\newcommand\Le{\epsilon}
\newcommand\lh{\eta}
\newcommand\li{\iota}
\newcommand\lk{\kappa}
\renewcommand\ll{\lambda}
\newcommand\lL{\Lambda}
\newcommand\lm{\mu}
\renewcommand\ln{\nu}
\newcommand\lf{\phi}
\newcommand\lF{\Phi}
\renewcommand\lq{\psi}
\newcommand\lQ{\Psi}
\renewcommand\#{\sharp}
\newcommand\lp{\pi}

\newcommand\ls{\sigma}

\newcommand\lo{\omega}

\newcommand\lx{\xi}
\newcommand\lz{\zeta}
\newcommand\iu{\cup}

\newcommand\ic{\subset}

\newcommand\xx{\times}
\newcommand\xt{\otimes}
\newcommand\oo{\infty}
\newcommand\oc{\circ}

\newcommand\AL{\mathcal A} 
 
\newcommand\CL{\mathcal C} 
 
\newcommand\EL{\mathcal E}

\newcommand\HL{\mathcal H} 
\newcommand\IL{\mathcal I} 
\newcommand\JL{\mathcal J} 
\newcommand\KL{\mathcal K} 
\newcommand\LL{\mathcal L}

\newcommand\OL{\mathcal O} 
\newcommand\PL{\mathcal P}

\newcommand\VL{\mathcal V}

\newcommand\Bl{{\mathbf B}} 
\newcommand\Cl{{\mathbf C}} 
\newcommand\Dl{{\mathbf D}}

\newcommand\Nl{{\mathbf N}}

\newcommand\Rl{{\mathbf R}}

\newcommand\gL{\frak g}
\newcommand\kL{\frak k}
\newcommand\pL{\frak p}
\newcommand\tL{\frak t}
\newcommand\cl{{\boldsymbol{c}}}
\newcommand\gl{{\boldsymbol{g}}}

\subjclass[2010]{47A13, 47B32, 22E46}
\keywords{The Bergman operator, the Jordan triple product, co-variant kernel, homogeneous operator, holomorphic Hermitian homogeneous vector bundles, intertwining operator, irreducible action, Wallach set}
\thanks{The work of the first author is supported through the IRSES Network, J C Bose National Fellowship of the DST and the UGC Centre for Advanced Studies.}

\address[G. Misra]{Department of Mathematics\\Indian Institutte of Science\\Bangalore 560012}
\address[Upmeier]{Department of Mathematics\\
University of Marburg\\
35032 Marburg}

\email[G. Misra]{gm@math.iisc.ernet.in}
\email[H. Upmeier]{upmeier@mathematik.uni-marburg.de}

\begin{document}
\title{Homogeneous Vector Bundles and intertwining Operators for Symmetric Domains}
\author{Gadadhar Misra and Harald Upmeier} 

\begin{abstract} The main features of homogeneous Cowen-Douglas operators, well-known for the unit disk, are generalized to the setting of hermitian bounded symmetric domains of arbitrary rank.
\end{abstract}

\maketitle
\setcounter{section}{-1}



\section{Introduction}
The well-known Cowen-Douglas theory of operators (resp. operator tuples) possessing an open set of (joint) eigenvalues, has been
implemented by A. Koranyi and G. Misra \cite{KM1,KM2} for the so-called homogeneous operators on the unit disk, leading to a full classification up to unitary equivalence. The unit disk is the basic example of a bounded hermitian symmetric domain. In this paper the major results for homogeneous operator tuples are generalized to all (irreducible) bounded symmetric domains $D=G/K$ of arbitrary rank.
More precisely, for the bi-holomorphic automorphism group $G$ of $D$ we realize some of the irreducible unitary (discrete series) representations on 'little' Hilbert spaces $H^2_\ln(D, P_\ll Z),$ for some $\ln >0$ and an integer partition $\ll$, consisting of holomorphic functions on the bounded symmetric domain $D$ and possessing a reproducing kernel. As our main construction, we then consider
a 'big' Hilbert space $H^2_{\ln,\cl}(D, P^nZ)$ with a unitary $G$-action, as a weighted direct sum with weights determined by a tuple $\cl$ of positive real numbers, which also possesses a covariant reproducing kernel. There is a natural action of the polynomial ring, or more generally, the $C^*$-algebra generated by the multiplication by polynomials and their adjoints, on both the little and the big Hilbert spaces. 

In this setting, our main results are (i) the 'Peter-Weyl' decomposition of the big Hilbert space into $G$-irreducible components described via intertwining differential operators, (ii) a criterion for boundedness of multiplication operators, and 
(iii) the proof of irreducibility of the associated homogeneous hermitian vector bundles under the multiplier action. Our results are new
even in the special case of the unit ball in $\Cl^d,$ where the general theory leads to explicit formulas for the intertwining operators and the classification of reproducing kernels.   

\section{Review of homogeneous Cowen-Douglas operators} 
To put our work into perspective, we first outline the relation to the standard approach to Cowen-Douglas theory. Consider an open subset $D\ic\Cl.$ In a very influential paper \cite{CD1}, Cowen and Douglas isolated a class of operators $B_k(D)$ of operators $T^*$ on a complex Hilbert space $\HL$ such that (i) for each $w\in D,$ the conjugate $\o w$ is an eigenvalue of constant multiplicity $k,$ (ii) the collection of all $k$-dimensional eigenspaces
$$\VL_w:=\ker(T^*-\o w)\ic\HL\qquad\forall\,w\in D$$
is total in $\HL,$ and (iii) given any $w_0\in D,$ there exists an anti-holomorphic choice of $k$ linearly independent eigenvectors 
\be{1}w\mapsto \KL_w^1,\ldots,\KL_w^k\ee
for $w$ in some neighborhood $U$ of $w_0.$ Putting 
$$\KL_w\lx=\S_{i=1}^k\lx_i\KL_w^i\qquad\forall\,\lx\in V:=\Cl^k$$ 
we obtain linear isomorphisms
\be{2}\KL_w:V\to\VL_w\qquad\forall\,w\in U\ee 
which depend anti-holomorphically on $w\in U.$ In the multi-variable case of domains $D\ic\Cl^d$, Curto and Salinas \cite{CS} introduced the class $B_k(D)$ of commuting $d$-tuples 
\be{3}(T_1^*,\ldots,T_d^*)\ee
of operators on $\HL,$ such that (i) for each $w\in D$ the adjoint tuple $\o w\in\Cl^d$ is a joint eigenvalue of constant multiplicity 
$k$ and (ii) the collection of all joint eigenspaces
$$\VL_w:=\bigcap_{j=1}^d\ker (T_j^*-\o w_j)\ic\HL\qquad\forall\,w\in D$$
is total in $\HL.$ Curto and Salinas list conditions on a commuting $d$-tuple which implies (iii) the existence of an anti-holomorphic choice \er{1} of $k$ linearly independent joint eigenvectors in some neighborhood $U$ of an arbitrary $w_0\in D.$ Under these conditions, the {\bf eigenbundle} 
$$\VL:=\U_{w\in D}\VL_w\ic D\xx\HL$$ 
of the operator tuple \er{3} becomes a {\bf hermitian holomorphic vector bundle} over $D.$ Thus every local holomorphic section $\ls:U\to\VL$ has a trivialization 
\be{4}\t\ls_z:=\KL_z^*\ls_z\in V\qquad\forall\,z\in U,\ee
which is a holomorphic map $\t\ls:U\to V.$ On the other hand, every $\lF\in\HL$ induces a global section
\be{5}z\mapsto\y\lF_z:=\li_z^*\lF\in\VL_z\qquad\forall\,z\in D\ee
of $\VL,$ where $\li_z:\VL_z\to\HL$ is the inclusion map. The local trivialization
\be{6}\t{\y\lF}_z=\KL_z^*\y\lF_z=\KL_z^*\li_z^*\lF=(\li_z\KL_z)^*\lF=\KL_z^*\lF\qquad\forall\,z\in U\ee
shows that these sections are holomorphic. Since the realization \er{5} is injective due to condition (ii), we may replace $\HL$ by a Hilbert space of holomorphic sections of $\VL.$ Then the $d$-tuple $T=(T_1,\ldots,T_d)$ is conjugate to the commuting tuple $z_j^\oc$ of multiplication operators by the coordinate functions $z_j$: In fact, $T_j^*\li_z=\o z_j\li_z$ implies
$$\y{T_j\lF}_z=\li_z^*T_j\lF=(T_j^*\li_z)^*\lF=(\o z_j\li_z)^*\lF=z_j(\li_z^*\lF)=z_j\cdot\y\lF_z.$$
More generally, denoting by $f^\oc$ the (module) action by multiplication operators, we have
\be{7}\y{f(T)\lF}_z=f(z)\y\lF_z=(f^\oc\lF)_z\ee
for all polynomials $f\in\Cl[z_1,\ldots,z_d].$ The papers \cite{CD1,CD2,CS} provide a 1-1 correspondence between unitary equivalence classes of operators in $B_k(D)$ and the corresponding equivalence classes of holomorphic Hermitian eigenbundles, with equivalence determined modulo a holomorphic change of frames.\\
In our setting of symmetric spaces of non-compact type, the eigenbundle has a {\bf global} holomorphic trivialization, so that \er{2}, \er{4} and \er{6} hold for $U=D.$ Thus we may identify $\HL$ with a Hilbert space of holomorphic maps $\lF:D\to V$ by identifying 
$\lF\in\HL$ with $\t{\y\lF}:D\to V.$ Define
$$\KL_{z,w}:=\KL_z^*\KL_w\in End(V)\qquad\forall\,z,w\in D.$$
For every $\lh\in V,$ the element $\KL_w\lh\in\VL_w\ic\HL$ induces the holomorphic section
$$\y{\KL_w\lh}_z=\li_z^*\KL_w\lh$$
with trivialization
$$\w{\y{\KL_w\lh}}_z=\KL^*_z\y{\KL_w\lh}_z=\KL^*_z\li_z^*\KL_w\lh=\KL^*_z\KL_w\lh=\KL_{z,w}\lh.$$
Identifying $\lF=\t{\y\lF}$ and $(\KL_w\lh)_z=\w{\y{\KL_w\lh}}_z=\KL_{z,w}\lh,$ it follows that $\KL_z$ becomes the adjoint of the evaluation map
$$\lF_z=\KL_z^*\lF\qquad\forall\, z\in D$$
and we have the {\bf reproducing kernel property} 
\be{8}(\lF_w|\lh)_V=(\t{\y\lF}_w|\lh)_V=(\KL_w^*\lF|\lh)_V=(\lF|\KL_w\lh)_\HL=(\lF|\w{\y{\KL_w\lh}})_\HL=(\lF|\KL_w\lh)_\HL\ee
for all $\lF\in\HL\ic\OL(D,V).$ We use Hilbert space inner products that are conjugate-linear in the first variable. Applying \er{8} to $\KL_z\lx\in\HL$ we obtain 
$$(\KL_z\lx|\KL_w\lh)_\HL=((\KL_z\lx)_w|\lh)_V=(\KL_{z,w}\lx|\lh)_V\qquad\forall\,z,w\in D,\,\forall\,\lx,\lh\in V.$$ 
Assuming that $\KL_{w,w}$ is invertible we have $\KL_w\KL_{w,w}^{-1}\KL_w^*=id_{\VL_w}.$ It follows that
$$(\y\lF_w|\y\lQ_w)_{\VL_w}=(\li_w^*\lF|\li_w^*\lQ)_\HL=(\li_w^*\lF|\KL_w\KL_{w,w}^{-1}\KL_w^*\li_w^*\lQ)_\HL$$
$$=(\KL_w^*\li_w^*\lF|\KL_{w,w}^{-1}\KL_w^*\li_w^*\lQ)_V=(\t{\y\lF}_w|\KL_{w,w}^{-1}\t{\y\lQ}_w)_V=(\lF_w|\KL_{w,w}^{-1}\lQ_w)_V$$
for all $\lF,\lQ\in\HL.$ Thus the hermitian fibre metric on the vector bundle $\VL$ is given by $\KL_{w,w}^{-1}.$ Then the $d$-tuple \er{3} is conjugate to the commuting tuple $z_j^\oc$ of multiplication operators by the coordinate functions
$$(T_j\lF)_z=z_j\,\lF_z=(z_j^\oc\lF)_z\qquad\forall\, z\in D.$$

In this paper we are concerned with {\bf equivariant} vector bundles. Let $G$ be a Lie group acting biholomorphically on $D,$ such that each $g\in G$ lifts to an automorphism of $\lp:\VL\to D,$ i.e. a diffeomorphism (linear on fibres) $\y g:\VL\to\VL$ such that $\lp\circ\y g=g\circ\lp,$ and $\y{g_1g_2}=\y g_1\y g_2$ for all $g_1,g_2\in G.$ Let $\y g_z:\VL_z\to\VL_{gz}$ denote the restriction of 
$\y g$ to $\VL_z,$ and consider the action $\y g^\la$ on holomorphic sections $\ls:D\to\VL$ defined by
$$(\y g^\la\,\ls)_{gz}=\y g_z\,\ls_z\qquad\forall\,z\in D.$$
Now assume that $\HL$ carries a (projective) unitary representation $g^\la$ of $G$ such that
$$\y{g^\la\lF}=\y g^\la\,\y\lF\qquad\forall\,\lF\in\HL.$$
Thus
$$\li_{gz}^*\,g^\la\,\lF=\y{g^\la\lF}_{gz}=(\y g^\la\y\lF)_{gz}=\y g_z\y\lF_z=\y g_z\,\li_z^*\lF,$$
showing that $\li_{gz}^*\,g^\la=\y g_z\,\li_z^*.$ Using the adjoint $\y g_z^*:\VL_{gz}\to\VL_z,$ we obtain
$$\li_{gz}=g^\la\,\li_z\,\y g_z^*.$$ 
This implies $g^\la\,\VL_z=\VL_{gz}.$ Since $\KL_{gz}$ and $g^\la\,\KL_z$ are both isomorphisms $V\to\VL_{gz},$ there exists a unique 'cocyle' $[g]_z\in GL(V)$ with the property
$$\KL_{gz}=g^\la\,\KL_z[g]_z^*.$$
Taking adjoints, we obtain the identity
$$(g^\la\lF)_{gz}=\KL_{gz}^*g^\la\lF=[g]_z\,\KL_z^*\lF=[g]_z\,\lF_z$$
for all $\lF\ic\OL(D,V).$ Thus the action $g^\la$ is a {\bf multiplier action}
$$(g^{-\la}\lF)_z=[g]_z^{-1}\,\lF_{gz}\qquad\forall\,z\in D.$$
Since $g^\la$ is unitary it follows that
\be{62}\KL_{gz,gw}=\KL_{gz}^*\KL_{gw}=(g^\la\,\KL_z[g]_z^*)^*g^\la\,\KL_w[g]_w^*=[g]_z\KL_z^*\KL_w[g]_w^*=[g]_z\KL_{z,w}[g]_w^*.\ee
Assuming that \er{7} holds for a larger algebra $\AL$ of holomorphic functions $f$ on $D$ which is invariant under the group action 
(e.g., all bounded holomorphic functions), we obtain the {\bf imprimitivity relations}    
$$f^\oc(g^\la\,\lF)=g^\la\,(f\circ g)^\oc\lF\qquad\forall\,f\in\AL.$$ 
In fact, for any $\lF\in\HL\ic\OL(D,V)$ and $z\in D$ we have
$$(f^\oc(g^\la\,\lF))_{gz}=f(gz)\,(g^\la\lF)_{gz}=(f\circ g)(z)\,[g]_z\,\lF_z=[g]_z((f\circ g)^\oc\lF)_z=(g^\la\,(f\circ g)^\oc\lF)_{gz}.$$
Another result of \cite{CD1} says that $T$ is irreducible if and only if the vector bundle $\VL$ is irreducible. In the homogeneous setting the Hermitian structure of the vector bundle $\VL$ may be irreducible, even if the action of the group of automorphisms on the sections of the vector bundle $\VL$ is reducible. When the module action is continuous, this is exactly the same as the irreducibility of the $d$-tuple of the adjoint of the multiplication operators on the sections of the vector bundle $\VL$. 

\section{Homogeneous vector bundles on symmetric Domains}
In this section we construct holomorphic cocycles and their associated homogeneous vector bundles on hermitian bounded symmetric domains of arbitrary rank. It is well known \cite{FK,L} that these domains can be algebraically described in terms of {\bf Jordan algebras} and {\bf Jordan triples}. Let $Z$ be an irreducible hermitian Jordan triple of rank $r,$ with Jordan triple product denoted by 
$(u\Box v^*)w=\{uv^*w\}$ for all $u,v,w\in Z.$ The {\bf quadratic representation} is defined by
\be{35}Q_u\,z:=\{uz^*u\}.\ee
Define the {\bf Bergman endomorphism} $B_{u,v}\in End (Z)$ by
$$B_{u,v}=I-2u\,\Box\,v^*+Q_u\,Q_v\qquad\forall\,u,v\in Z.$$ 
Let $K\ic GL(Z)$ denote the compact linear group of all Jordan triple automorphisms. The complexified group $C=K^\Cl$ consists of all invertible linear maps $h\in GL(Z)$ satisfying 
\be{16}h\,Q_z\,h^*=Q_{hz}\qquad\forall\,z\in Z.\ee 
Here $h^*$ is the adjoint relative to the $K$-invariant inner product
$$(u|v):=\f2p\, tr\, u\Box v^*,$$
where $p$ is the so-called genus. We have 
$$h^{-1}\,B_{u,v}\,h=B_{h^{-1}u,\,h^*v}$$  
for all $h\in K^\Cl$. The (spectral) unit ball 
$$D=\{z\in Z:\,B(z,z)>0\}\ic Z$$ 
is an irreducible hermitian bounded symmetric domain. The identity component $G$ of the biholomorphic automorphism group of 
$D$ is a semi-simple real Lie group and $D=G/K,$ with $K=\{g\in G|\,g(0)=0\}$ the stabilizer at the origin $0\in D$. A pair $(z,w)\in Z\xx Z$ is called {\bf quasi-invertible} if $B_{z,w}\in GL(Z)$ is invertible. In this case, the element
$$z^w:=B^{-1}_{z,w}(z-Q_z\, w)\in Z$$  
is called the {\bf quasi-inverse}. Its exact denominator is the {\bf quasi-determinant} $\lD_{z,w},$ an (inhomogeneous) sesqui-polynomial which satisfies
$$\det B_{z,w}=\lD_{z,w}^p.$$
The covariance formula 
\be{9}B_{gz,gw}=(\dl_zg)\,B_{z,w}\,(\dl_wg)^*\qquad\forall\,g\in G\ee
implies
\be{10}\lD_{gz,gw}=\det(\dl_zg)^{1/p}\,\lD_{z,w}\,\o{\det(\dl_wg)}^{1/p}\qquad\forall\,g\in G.\ee
Therefore the $G$-invariant measures on $D$ have the form
\be{15}d\lm(x)=c\cdot\,\lD_{x,x}^{-p}\,dx,\ee
where $c>0$ is a constant and $dx$ denotes Lebesgue measure on $Z.$ By \cite{L} the compact dual space $\h Z$ of $Z,$ called the {\bf conformal compactification} of $Z$ has a Jordan theoretic description in terms of the quasi-inverse. The group $G^\Cl$ of all biholomorphic automorphisms of $\h Z$ has a decomposition
$$G^\Cl=P_-\,K^\Cl\,P_+,$$  
where $P_-$ consists of all translations
$$\tL_w\,z:=z+w,\,w\in Z$$  
and $P_+$ consists of all rational transformations
$$\o\tL_w\,z:=z^{-w},\,w\in Z.$$  
Then $\h Z=G^\Cl/K^\Cl\,P_+.$ For all $\,h\in K^\Cl$ we have 
\be{50}h\,\tL_w\,h^{-1}=\tL_{hw}\ee 
and $h^{-1}(z^w)=(h^{-1}\,z)^{h^*w},$ or equivalently
\be{29}h^{-1}\,\o\tL_w\,h=\o\tL_{h^*w}.\ee 
One can show \cite{L} that 
\be{51}\o\tL_w\,\tL_z=\tL_{z^{-w}}\,B^{-1}_{z,-w}\,\o\tL_{w^{-z}}\ee
whenever $(z,-w)$ is quasi-invertible. For any $x\in D$ we define the {\bf transvection}
\be{21}\gl_x:=\tL_x\,B^{1/2}_{x,x}\,\o\tL_x\in G\ee
using the sesqui-holomorphic square-root $(z,w)\mapsto B_{z,w}^{1/2}$ on the simply-connected domain $D\xx D$ determined by 
$B_{0,0}^{1/2}=id.$
\begin{definition} Let $C$ be a complex Lie subgroup of $G^\Cl$ containing $K^\Cl$. A mapping 
$$G\xx D\to C,\,(g,z)\mapsto [g]_z,$$
is called a $C$-valued {\bf holomorphic cocycle} if for every $g\in G$ the map $D\ni z\mapsto [g]_z\in C$ is holomorphic and the cocycle identity
\be{13}[g_1\,g_2]_z=[g_1]_{g_2(z)}\,[g_2]_z\ee
holds for all $g_1,g_2\in G$ and $z\in D.$ We assume in addition that
\be{14}[k]_z=k\qquad\forall\,k\in K.\ee 
\end{definition}
\begin{definition}\label{a} Consider a holomorphic $C$-valued cocycle and an irreducible holomorphic representation of $C$ on a complex vector space $V,$ endowed with a $K$-invariant inner product $(\lf|\lq)_V.$ The $V$-{\bf valued Bergman space} $H^2_\ln(D,V)$ consists of all holomorphic maps $D\ni z\mapsto\lF_z\in V$ which are square-integrable under the inner product
\be{11}(\lF|\lQ)=\I_{D}d\lm(x)\lD_{x,x}^\ln\,([\gl_x]_0^{-1}\,\lF_x|\,[\gl_x]_0^{-1}\,\lQ_x)_V.\ee
Here $\ln$ is a parameter which is sufficiently large to ensure that $H^2_\ln(D,V)$ contains all polynomial maps in the variable $z.$
\end{definition}
In this definition the inner product is given in terms of a measure on $D.$ Later, we will consider more general parameters $\ln$ in the so-called (continuous) Wallach set, which are obtained via analytic continuation. For some special values of $\ln$, the corresponding inner products may still be defined via semi-invariant measures on boundary orbits of $G,$ but in general there is no measure theoretic realization anymore.
\begin{proposition}\label{a} For a holomorphic map $\lF:D\to V$ we put
$$(g^{-\la}\,\lF)_z=[g]_z^{-1}\cdot\,\lF_{gz}\qquad\forall\,(g,z)\in G\xx D.$$
Then
\be{12}(g^{-\la_\ln}\,\lF)_z=\det(\dl_zg)^{\ln/p}\,(g^{-\la}\lF)_z\ee
defines a unitary (projective) representation of $G$ on $H^2_\ln(D,V).$ 
\end{proposition}
\begin{proof} Clearly, $g^{-\la_\ln}\,\lF:D\to V$ is holomorphic, and we have $g_1^{\la_\ln}g_2^{\la_\ln}=(g_1\,g_2)^{\la_\ln}$ whenever $g_1,g_2\in G$, since \er{13} implies
$$(g_2^{-\la_\ln}\,g_1^{-\la_\ln}\,\lF)_z=\det(\dl_zg_2)^{\ln/p}\,[g_2]_z^{-1}\cdot(g_1^{-\la_\ln}\,\lF)_{g_2(z)}$$
$$=\det(\dl_zg_2)^{\ln/p}\,[g_2]_z^{-1}\cdot\det(\dl_{g_2(z)}g_1)^{\ln/p}\,[g_1]_{g_2(z)}^{-1}\,\lF_{g_1(g_2(z))}
=\det(\dl_z(g_1\,g_2))^{\ln/p}\,[g_1\,g_2]_z^{-1}\,\lF_{(g_1\,g_2)(z)}.$$
More precisely, this identity holds up to a constant multiple depending on $g_1,g_2$ if $\ln/p$ is not an integer. To show unitarity, note that for any $x\in D$ and $g\in G$, we have $g\,\gl_x=\gl_{gx}k$ for a suitable $k\in K$. It follows that
$$[g]_x\,[\gl_x]_0=[g\,\gl_x]_0=[\gl_{gx}k]_0=[\gl_{gx}]_0\,[k]_0=[\gl_{gx}]_0\,k$$  
by \er{13} and \er{14}. By $K$-invariance we obtain
$$\([\gl_x]_0^{-1}\,[g]_x^{-1}\,\lf\|\,[\gl_x]_0^{-1}\,[g]_x^{-1}\,\lq\)_V
=\(k^{-1}\,[\gl_{gx}]_0^{-1}\,\lf\|\,k^{-1}\,[\gl_{gx}]_0^{-1}\,\lq\)_V=\([\gl_{gx}]_0^{-1}\,\lf\|\,[\gl_{gx}]_0^{-1}\,\lq\)_V$$
for all $\lf,\lq\in V$. Using the $G$-invariant measure \er{15}, we obtain
$$\(g^{-\la_\ln}\,\lF\|g^{-\la_\ln}\,\lQ\)=\I_{D}d\lm(x)\,\lD^\ln_{x,x}\,
\([\gl_x]_0^{-1}\,(g^{-\la_\ln}\,\lF)_x\|\,[\gl_x]_0^{-1}\,(g^{-\la_\ln}\,\lQ)_x\)_V$$
$$=\I_{D}d\lm(x)\,\lD^\ln_{x,x}\,|\det(\dl_xg)|^{2\ln/p}\,
\([\gl_x]_0^{-1}\,[g]_x^{-1}\,\lF_{gx}\|\,[\gl_x]_0^{-1}\,[g]_x^{-1}\,\lQ_{gx}\)_V$$
$$=\I_{D}d\lm(x)\,\lD^\ln_{gx,gx}\,\([\gl_{gx}]_0^{-1}\,\lF_{gx}|\,[\gl_{gx}]_0^{-1}\,\lQ_{gx}\)_V 
=\I_{D}d\lm(x)\,\lD^\ln_{x,x}\,\([\gl_x]_0^{-1}\,\lF_x\|\,[\gl_x]_0^{-1}\,\lQ_x\)_V=(\lF|\lQ).$$
\end{proof}
The distinction between $g^\la$ and $g^{\la_\ln}$ made in the previous proposition isolates the role of the 'quantization' or 'deformation' parameter $\ln.$ Alternatively, one could adapt the cocycle $[g]_z$ to incorporate the parameter 
$\ln.$ Geometrically, this means tensoring $V$ by a suitable power of a line bundle. 

Our first example concerns the group $C=K^\Cl.$
\begin{proposition} For every $g\in G$, the complex derivative $\dl_zg\in GL(Z)$ at $z\in D$ belongs to  $K^\Cl,$ giving a holomorphic cocycle
$$G\xx D\to K^\Cl,\,(g,z)\mapsto\dl_zg.$$   
\end{proposition}
\begin{proof} In view of \er{16}, the identity \cite{L}
$$B_{u,v}\,Q_z\,B_{v,u}=Q_{B_{u,v}z}$$  
shows that $B_{u,v}\in K^\Cl$ whenever $(u,v)$ is quasi-invertible. Since $\o\tL_x(0)=0$ and $B^{1/2}_{x,x}$ is linear, we have
$\gl_x(0)=\tL_x(0)=x.$ Since
$$\dl_z\o\tL_x=B^{-1}_{z,-x}$$  
it follows that
\be{17}\dl_z\gl_x=B^{1/2}_{x,x}\,B^{-1}_{z,-x}\ee
belongs to $K^\Cl$ for all $z,x\in D.$ In particular,
\be{18}\dl_0\gl_x=B^{1/2}_{x,x}.\ee
Since $\dl_zk=k$ for all $k\in K,$ the first assertion follows. The chain rule implies the cocycle identity 
$\dl_z(g_1\,g_2)=(\dl_{g_2z}g_1)(\dl_zg_2)$ for all $g_1,g_2\in G$ and $z\in D.$
\end{proof}
In order to find irreducible holomorphic representations of $K^\Cl,$ consider the algebra $\PL Z$ of all holomorphic polynomials 
$\lf:Z\to\Cl,$ with the holomorphic action $\lp$ defined by 
$$(h^{-\lp}\,\lf)(\lz):=\lf(h\lz)\qquad\forall\,h\in K^\Cl.$$ 
By \cite{S,U3} there is a multiplicity-free {\bf Peter-Weyl decomposition}
$$\PL Z=\S_\ll\PL_\ll Z$$
into irreducible $K^\Cl$-submodules $\PL_\ll Z$, which are pairwise inequivalent. Here $\ll$ runs over all {\bf integer partitions}
$$\ll_1\ge\cdots\ge\ll_r\ge 0$$  
of length $\le r$. Using the {\bf Fischer-Fock inner product}
$$(\lf|\lq)_F:=\I_Z\f{d\lo}{\lp^d}\,e^{-(\lo|\lo)}\,\o{\lf(\lo)}\,\lq(\lo),$$
it follows that there is an expansion 
$$e^{(\lz|\lo)}=\S_\ll\EL^\ll_{\lz,\lo}\qquad\forall\,\lz,\lo\in Z,$$ 
where
\be{60}\EL^\ll_{\lz,\lo}=\S_m\lf_\ll^m(\lz)\o{\lf_\ll^m(\lo)}\ee
is the reproducing kernel of $\PL_\ll Z,$ and $\lf_\ll^m$ is any orthonormal basis of $\PL_\ll Z.$
\begin{definition} For a fixed partition $\ll,$ the {\bf little Hilbert space} $H^2_\ln(D,\,\PL_\ll Z)$ consists of all holomorphic maps $D\ni z\mapsto\lF_z(\lz)\in\PL_\ll Z$ which are square-integrable under the inner product 
\be{63}(\lF|\lQ):=\I_{D}d\lm(x)\,\lD_{x,x}^\ln\,(\lF_x|\lQ_x\circ B_{x,x})_F.\ee
Here the parameter $\ln$ is sufficiently large to ensure that $H^2_\ln(D,\PL_\ll Z)$ contains all polynomial maps in the variable $z.$
\end{definition}
\begin{proposition} For a holomorphic map $\lF:D\to\PL_\ll Z$ we put
$$(g^{-\la}\,\lF)_z (\lz)=\lF_{g(z)}((\dl_zg)\lz)\qquad\forall\,(g,z)\in G\xx D.$$
Then
\be{19}(g^{-\la_\ln}\,\lF)_z(\lz)=\det(\dl_zg)^{\ln/p}\,(g^{-\la}\lF)_z(\lz)=\det(\dl_zg)^{\ln/p}\,\lF_{g(z)}(\dl_zg)\lz)\ee
defines a unitary (projective) representation of $G$ on $H^2_\ln(D,\,\PL_\ll Z).$
\end{proposition}
\begin{proof} Since $B_{x,x}\in K^\Cl$ is self-adjoint, we have
$$(\lf\circ(\dl_0\gl_x)|\lq\circ(\dl_0\gl_x))_F=(\lf\circ B^{1/2}_{x,x}|\lq\circ B^{1/2}_{x,x})_F=(\lf|\lq\circ B_{x,x})_F$$
for $\lf,\lq\in\PL_\ll Z.$ Now the assertion follows as a special case of Proposition \er{a}.
\end{proof}
The representation \er{19} is irreducible and belongs to the {\bf holomorphic discrete series} of $G$ if $\ln$ is sufficiently large.
For the 'empty' partition $\ll=0$ we may identify $\PL_0Z=\Cl$ with the constant functions. Thus we obtain the usual scalar-valued weighted Bergman space $H^2_\ln(D,\Cl).$
\begin{example} For the unit disk $\Dl$ and $\ll\in\Nl$, $H^2_\ln(\Dl,\PL_\ll\Cl)$ consists of holomorphic maps
$$z\mapsto p_z(\lz)=\lz^\ll\,p(z)$$  
on $\Dl$, with the $G$-action  
$$(g^{-\la_\ln}\,p)_z(\lz)=(\dl_zg)^{\ln/2}\,p(gz)((\dl_zg)\lz)^\ll=(\dl_zg)^{\ll+\ln/2}\,p(gz)\,\lz^\ll
=\f{\lz^\ll}{(cz+d)^{2\ll+\ln}}\,p(g(z)).$$ 
It follows that
$$(g^{-\la_\ln}\,p)(z)=\f{p(g(z))}{(cz+d)^{2\ll+\ln}}.$$  
\end{example}
Our main construction involves the group $C=K^\Cl\,P_+.$ 
\begin{proposition} The translations $z\mapsto\tL_z$ define a holomorphic map $D\to G^\Cl,$ and
\be{20}[g]_z:=\tL^{-1}_{g(z)}\,g\,\tL_z=(\dl_zg)\o\tL_{-g^{-1}(0)^z}\ee
is a holomorphic cocycle with values in $K^\Cl\,P_+.$
\end{proposition} 
\begin{proof} The cocycle identity
$$[g_1]_{g_2 z}\,[g_2]_z=\tL^{-1}_{g_1\,(g_2 z)}\,g_1\,\tL_{g_2 z}\,\tL^{-1}_{g_2 z}\,g_2\,\tL_z
=\tL^{-1}_{(g_1\,g_2)(z)}\,g_1\,g_2\,\tL_z=[g_1\,g_2]_z$$  
follows from the definition. For $k\in K$ we have $[k]_z=\tL^{-1}_{kz}\,k\,\tL_z=k.$ For the proof of \er{20} note that for the transvections \er{21} we have, in view of \er{50}, \er{51} and \er{17},
$$[\gl_x]_z=\tL^{-1}_{\gl_x(z)}\,\gl_x\,\tL_z=\tL^{-1}_{\gl_x(z)}\,\tL_x\,B^{1/2}_{x,x}\,\o\tL_x\,\tL_z=\tL^{-1}_{\gl_x(z)-x}\,B^{1/2}_{x,x}\,\o\tL_x\,\tL_z=\tL^{-1}_{B^{1/2}_{x,x}(z^{-x})}\,B^{1/2}_{x,x}\,\o\tL_x\,\tL_z$$
$$=B^{1/2}_{x,x}\,\tL^{-1}_{z^{-x}}\,\o\tL_x\,\tL_z=B^{1/2}_{x,x}\,\tL^{-1}_{z^{-x}}\,\tL_{z^{-x}}\,B^{-1}_{z,-x}\,\o\tL_{x^{-z}}
=B^{1/2}_{x,x}\,B^{-1}_{z,-x}\,\o\tL_{x^{-z}}=(\dl_z\gl_x)\,\o\tL_{x^{-z}}.$$  
This shows
$$[\gl_x]_z=(\dl_z\gl_x)\,\o\tL_{x^{-z}}.$$  
In particular
\be{25}[\gl_x]_0=B^{1/2}_{x,x}\,\o\tL_x.\ee
For the general case, write $g=\gl_x\,k$, where $x=g(0)$ and $k\in K$. Then the cocycle identity implies
$$[g]_z=[\gl_x\,k]_z=[\gl_x]_{kz}\,[k]_z=[\gl_x]_{kz}\,k=(\dl_{kz}\gl_x)\o\tL_{x^{-kz}}\,k$$
$$=\dl_z(\gl_x\,k)\,k^{-1}\,\o\tL_{x^{-kz}}\,k=(\dl_zg)\,\o\tL_{k^{-1}(x^{-kz})}=(\dl_zg)\,\o\tL_{(k^{-1}x)^{-z}}$$  
with $k^{-1}x=-k^{-1}\,\gl^{-1}_x(0)=-(\gl_x\,k)^{-1}(0)=-g^{-1}(0).$
\end{proof}
An irreducible representation of $G^\Cl,$ in fact of $K^\Cl P_+,$ is obtained as follows. For fixed $n\in\Nl$ the polynomials
$$\lD^n_{-\lo}(\lz):=\lD^n_{\lz,-\lo},$$  
where $\lo\in Z$ is arbitrary, span a finite-dimensional subspace $\PL^nZ\ic\PL Z.$ There exists a holomorphic representation $\lp_n$ of 
$G^\Cl$ on $\PL^nZ$ defined by
\be{22}(\lg^{-\lp_n}\lf)(\lz)=\det(\dl_\lz\lg)^{-n/p}\,\lf(\lg(\lz))\ee  
for all $\lg\in G^\Cl,\lz\in Z$ and $\lf\in\PL^nZ.$ In analogy to the group $SL(2,\Cl),$ we call this the {\bf spin representation} of level $n$. Now consider the multi-variable Pochhammer symbol \cite{FK}
$$(\ln)_\ll:=\P^r_{j=1}\,(\ln-\f a2\,(j-1))_{\ll_j}$$  
and the Faraut-Kor\'anyi formula \cite{FK}
\be{23}\lD^{-\ln}_{\lz,\lo}=\S_\ll(\ln)_\ll\,\EL^\ll_{\lz,\lo}.\ee
Since $(-n)_\ll$ vanishes whenever $\ll_1>n$ we obtain the Peter-Weyl decomposition 
\be{26}\PL^nZ=\S_{\ll\le n}\PL_\ll Z\ee
labelled by the $\b{n+r}r$ partitions $\ll$ of length $\le r$ with largest part $\ll_1\le n.$ The general $K$-invariant inner product on $\PL^nZ$ is therefore given by
$$(\lf|\lq)=\S_{\ll\le n}b_\ll(\lf_\ll|\lq_\ll)_F$$
for arbitrary constants $b_\ll>0.$ Here we write 
\be{61}\lf=\S_{\ll\le n}\lf_\ll,\ee
where $\lf_\ll\in\PL_\ll Z$ is the Peter-Weyl component. Every $h\in K^\Cl$ acts via a diagonal operator. For $b_\ll=1$ we obtain the Fock inner product on $\PL^nZ\ic\PL Z.$ On the other hand, the inner product on $\PL^nZ,$ determined by
$$(\lD^n_{-\lz}|\lD^n_{-\lo})_U=\lD_{\lz,-\lo}^n\qquad\forall\,\lz,\lo\in Z,$$
is even invariant under the compact form $U$ of $G^\Cl,$ consisting of all biholomorphic isometries of $\h Z.$ This inner product is obtained by putting
$$b_\ll=(-1)^{|\ll|}(-n)_\ll=\P^r_{j=1}(-1)^{\ll_j}(-n-\f a2\,(j-1))_{\ll_j}=\P^r_{j=1}(n-\ll_j+\f a2\,(j-1)+1)_{\ll_j}.$$
Note that $n-\ll_j\ge 0$ by assumption. In fact, by \er{23} we have
$$\lD_{-\lz}^n=\S_{\ll\le n}(-n)_\ll\,\EL^\ll_{-\lz}=\S_{\ll\le n}(-1)^{|\ll|}(-n)_\ll\,\EL^\ll_\lz.$$
Using orthogonality this implies
$$\S_{\ll\le n}(-1)^{|\ll|}(-n)_\ll\,\EL^\ll_{\lz,\lo}=\lD_{\lz,-\lo}^n=(\lD_{-\lz}^n|\lD_{-\lo}^n)_U
=\S_{\ll\le n}\S_{\lm\le n}(-1)^{|\ll|}(-n)_\ll(-1)^{|\lm|}(-n)_\lm\,(\EL^\ll_\lz|\EL^\lm_\lo)_U$$
$$=\S_{\ll\le n}(-n)_\ll^2\,(\EL^\ll_\lz|\EL^\ll_\lo)_U=\S_{\ll\le n}b_\ll(-n)_\ll^2\,(\EL^\ll_\lz|\EL^\ll_\lo)_F
=\S_{\ll\le n}b_\ll(-n)_\ll^2\,\EL^\ll_{\lz,\lo}$$
showing that $b_\ll(-n)_\ll=(-1)^{|\ll|}.$
\begin{definition} Let $(\lf|\lq)$ be any $K$-invariant inner product on $\PL^nZ.$ The {\bf big Hilbert space} $H^2_\ln(D,\PL^nZ)$ consists of all holomorphic maps $D\ni z\mapsto\lF_z(\lz)\in\PL^nZ$ which are square-integrable under the inner product
$$(\lF|\lQ)=\I_{D}d\lm(x)\,\lD_{x,x}^\ln\,(\lD_{\lz,-x}^n\lF_x(B_{x,x}^{1/2}\lz^{-x})|\lD_{\lz,-x}^n\lQ_x(B_{x,x}^{1/2}\lz^{-x})).$$
Here the parameter $\ln$ is sufficiently large to ensure that $H^2_\ln(D,\PL^nZ)$ contains all polynomial maps in the variable $z.$
\end{definition}
\begin{proposition} For holomorphic maps $\lF:D\to\PL^nZ$ we put
$$(g^{-\lb}\,\lF)_z(\lz):=[g]_z^{-1}\,\lF_{g(z)}(\lz)=\det(\dl_{z+\lz}g)^{-n/p}\,\lF_{g(z)}(g(z+\lz)-g(z))\qquad\forall\,(g,z)\in G\xx D.$$ 
Then
\be{24}(g^{-\lb_\ln}\,\lF)_z(\lz):=\det(\dl_zg)^{(\ln+n)/p}\,(g^{-\lb}\lF)_z(\lz)=\det(\dl_zg)^{(\ln+n)/p}\,\det(\dl_{z+\lz}g)^{-n/p}\,\lF_{g(z)}(g(z+\lz)-g(z))\ee
defines a (projective) unitary representation of $G$ on $H^2_\ln(D,\PL^nZ).$
\end{proposition}
\begin{proof} By definition we have
$$[g]_z(\lz)=\tL^{-1}_{g(z)}\,g\,\tL_z(\lz)=\tL^{-1}_{g(z)}\,g(z+\lz)=g(z+\lz)-g(z)$$ 
and therefore 
$$\dl_\lz[g]_z=\dl_{z+\lz}g$$  
for all $g\in G,z\in D$ and $\lz\in Z$. It follows that
$$([g]_z^{-1}\lf)(\lz)=\det(\dl_\lz[g]_z)^{-n/p}\,\lf([g]_z(\lz))=\det(\dl_{z+\lz}g)^{-n/p}\,\lf(g(z+\lz)-g(z)).$$
for any $\lf\in\PL^nZ$. Since $[\gl_x]_0=B_{x,x}^{1/2}\o\tL_x$ by \er{25} it follows that 
$\dl_\lz[\gl_x]_0=B_{x,x}^{1/2}(\dl_\lz\o\tL_x)=B_{x,x}^{1/2}B_{\lz,-x}^{-1}$ and we obtain
$$([\gl_x]_0^{-1}\lq)(\lz)=\(\det B_{x,x}^{1/2}\,B_{\lz,-x}^{-1}\)^{-n/p}\,\lq(B_{x,x}^{1/2}\o\tL_x(\lz))
=\lD_{x,x}^{-n/2}\,\lD_{\lz,-x}^n\,\lq(B_{x,x}^{1/2}\lz^{-x}).$$
Now the assertion follows as a special case of Proposition \er{a}.   
\end{proof}
\begin{example} The unit disk $\Dl\ic\Cl$ has genus $p=2$ and $H^2_\ln(\Dl,\PL^n\Cl)$ consists of holomorphic maps 
$$\Dl\xx\Cl\xrightarrow{\lF}\Cl,\,(z,\lz)\mapsto\lF_z(\lz)$$  
such that $\lF_z(\lz)$ is a polynomial of degree $\le n$ in $\lz$. The action of $g=\bb abcd\in G=SU(1,1)$ is  
$$(g^{-\lb_\ln}\,\lF)_z(\lz)=(\dl_zg)^{(\ln+n)/2}\,(\dl_{z+\lz}g)^{-n/2}\,\lF_{g(z)}\,(g(z+\lz)-g(z))$$
$$=\f{(c(z+\lz)+d)^n}{(cz+d)^{\ln+n}}\,\lF_{g(z)}\,\(\f{\lz}{(c(z+\lz)+d)(cz+d)}\).$$
Since $\lF_{g(z)}$ is a polynomial of degree $\le n,$ this action yields again polynomials in $\lz$ of degree $\le n.$ 
\end{example}
Our first result concerning the cocycle \er{20} generalizes the canonical realization in terms of certain 'shift operators' in the $1$-dimensional case \cite{KM1}. We first define shift operators in our multi-dimensional setting. The Peter-Weyl decomposition \er{26} implies that any operator $T$ on $\PL^nZ$ is described by an operator matrix $T^\ll_\lm:\PL_\ll Z\to\PL_\lm Z,$ indexed by partitions $\ll,\lm\le n,$ via
$$(T\lf)_\lm=\S_{\ll\le n}T^\ll_\lm\,\lf_\ll\qquad\forall\,\lm\le n.$$
Using the partial ordering of partitions under the inclusion relation, we call the matrix $(T^\ll_\lm)$ {\bf lower-triangular} if the non-zero entries $T^\ll_\lm$ occur only for $\ll\ic\lm.$ The matrix will be called a {\bf shift} if the only nonzero entries are of the form $T^\ll_{\ll+\Le_i}$ for some $1\le i\le r,$ where $\Le_i$ is the $i$-th unit vector and $\ll+\Le_i$ is still a partition. By \cite{U1} these shift operators occur naturally in our setting: For fixed $w\in Z,$ the {\bf multiplication operator}
$$\lf(\lz)\mapsto (\lz|w)\,\lf(\lz)$$ 
and the {\bf Jordan differentiation operator} 
$$\lf(\lz)\mapsto(\dl_\lz\lf)Q_\lz w$$
are shift operators of this kind. Thus there exist linear operators $M^\ll_i(w),D^\ll_i(w)$ from $\PL_\ll Z$ to $\PL_{\ll+\Le_i}Z,$ depending anti-linearly on $w\in Z$ such that for all $\lf\in\PL_\ll Z$ we have
$$(\lz|w)\,\lf(\lz)=\S_{i=1}^r(M^\ll_i(w)\,\lf)(\lz),$$
$$(\dl_\lz\lf)Q_\lz w=\S_{i=1}^r(D^\ll_i(w)\,\lf)(\lz).$$
\begin{theorem}\label{d} For $g\in G$ the cocycle \er{20} has a factorization
$$[g]_z=(\dl_zg)\,B_{z,g^{-1}(0)}^{1/2}\,\exp S(g^{-1}(0))\,B_{z,g^{-1}(0)}^{-1/2}$$
where $\dl_zg$ and $B_{z,w}^{1/2}$ are diagonal operators, and $S(w)$ is the 'shift' operator matrix with non-zero entries 
$$S(w)^\ll_{\ll+\Le_i}=n\,M^\ll_i(w)-D^\ll_i(w).$$
\end{theorem}
\begin{proof} Since $w^{-z}=B^{-1/2}_{w,-z}w,$ we have $\o\tL_{w^{-z}}=\o\tL_{B^{-1/2}_{w,-z}w}=B^{1/2}_{z,-w}\,\o\tL_w\,B^{-1/2}_{z,-w},$ using \er{29}. Putting $w=-g^{-1}(0),$ we obtain with \er{20} the identity
$$[g]_z=(\dl_zg)\,B_{z,g^{-1}(0)}^{1/2}\,\o\tL_{-g^{-1}(0)}\,B_{z,g^{-1}(0)}^{-1/2}.$$
By \cite{L} the Lie algebra $\gL^\Cl$ of $G^\Cl$ consists of certain polynomial maps $\lg:Z\to Z$ of degree 
$\le 2.$ Using the quadratic representation \er{35} we have 
$$\o\tL_w=\exp(Q_\lz w)\qquad\forall\,w\in Z.$$
For the action \er{22}, this implies
$$\o\tL_w^{\lp_n}=\exp((Q_\lz w)^{\d\lp_n}),$$
where $\d\lp_n$ denotes the infinitesimal action of the Lie algebra $\gL^\Cl.$ Thus it suffices to compute $(Q_\lz w)^{\d\lp_n}.$  Differentiating \er{22} yields
$$\lg^{\d\lp_n}\lf(\lz)=(\dl_\lz\lf)\lg(\lz)-\f np\,tr(\dl_\lz\lg)\,\lf(\lz).$$
Since $tr(\dl_\lz\,Q_\lz w)=2\,tr\,\lz\Box w^*=p(\lz|w)$ we obtain
$$((Q_\lz w)^{\d\lp_n}\lf)(\lz)=(\dl_\lz\lf)Q_\lz w-n(\lz|w)\,\lf(\lz).$$
For $\lf\in\PL_\ll Z$ it follows that
$$((Q_\lz w)^{\d\lp_n}\lf)(\lz)=\S_{i=1}^r(D_\ll^i(w)-n\,M_\ll^i(w))\lf$$
with $(D^\ll_i(w)-n\,M^\ll_i(w))\lf\in\PL_{\ll+\Le_i}Z.$
\end{proof}

\section{Intertwining operators}
For a deeper study of the 'big' Hilbert space, e.g. the classification of covariant reproducing kernels, we need its decomposition into 
the $\b{n+r}{r}$ inequivalent isotypic components under the representation \er{24}. For each partition $\ll\le n,$ we will describe the associated components via {\bf intertwining operators} which are unique up to a constant multiple. Related problems are also studied in \cite{PZ} under the aspect of decomposition of tensor products of $G$-representations. We will describe these intertwiners both in terms of integral and differential operators.
\begin{theorem} For each partition $\ll\le n$ an intertwining operator $\IL_\ll:H^2_\ln(D,\PL_\ll Z)\to H^2_\ln(D,\PL^nZ)$
is given by
\be{32}(\IL_\ll\lF)_z(\lz)=\I_{D}d\lm(x)\,\f{\lD_{x,x}^\ln\,\lD_{z+\lz,x}^n}{\lD_{z,x}^{\ln+n}}\,\lF_x(B_{x,x}((z+\lz)^x-z^x)).\ee
\end{theorem}
\begin{proof} Since $\IL_\ll$ is $K$-invariant, it suffices to consider the transvections $\gl_y$ with $y\in D.$ For $z,x\in D$ we have
$$B_{y,y}^{1/2}z^{\gl_y(x)}=B_{y,y}^{1/2}z^{y+B_{y,y}^{1/2}x^{-y}}=B_{y,y}^{1/2}(z^y)^{B_{y,y}^{1/2}x^{-y}}
=(B_{y,y}^{1/2}z^y)^{B_{y,y}^{-1/2}B_{y,y}^{1/2}x^{-y}}=(B_{y,y}^{1/2}z^y)^{(x^{-y})}.$$
Applying the 'addition formula' \cite{L}
$$(u+v)^x=u^x+B_{u,x}^{-1}(v^{(x^u)})$$
to $u=-y$ and $v=B_{y,y}^{1/2}z^y$ it follows that
$$B_{y,-x}^{-1}B_{y,y}^{1/2}z^{\gl_y(x)}=B_{y,-x}^{-1}((B_{y,y}^{1/2}z^y)^{(x^{-y})})=((B_{y,y}^{1/2}z^y-y)^x-(-y)^x=
(\gl_{-y}(z))^x-(-y)^x.$$
As a consequence we obtain
$$\gl_{-y}(z+\lz)^x-\gl_{-y}(z)^x=B_{y,-x}^{-1}B_{y,y}^{1/2}((z+\lz)^{\gl_y(x)}-z^{\gl_y(x)}).$$
On the 'big' Hilbert space the action is 
$$(\gl_y^{\lb_\ln}\lQ)_z(\lz)=\f{\lD_{y,y}^{\ln/2}\,\lD_{z+\lz,y}^n}{\lD_{z,y}^{\ln+n}}\,\lQ_{\gl_{-y}(z)}(\gl_{-y}(z+\lz)-\gl_{-y}(z)).$$
Therefore
$$(\gl_y^{\lb_\ln}\IL_\ll\lF)_z(\lz)=\f{\lD_{y,y}^{\ln/2}\,\lD_{z+\lz,y}^n}{\lD_{z,y}^{\ln+n}}\,(\IL_\ll\lF)_{\gl_{-y}(z)}
(\gl_{-y}(z+\lz)-\gl_{-y}(z))$$
$$=\f{\lD_{y,y}^{\ln/2}\,\lD_{z+\lz,y}^n}{\lD_{z,y}^{\ln+n}}\,\I_{D}d\lm(x)\,
\f{\lD_{x,x}^\ln\,\lD_{\gl_{-y}(z+\lz),x}^n}{\lD_{\gl_{-y}(z),x}^{\ln+n}}\,(\lF_x\circ B_{x,x})(\gl_{-y}(z+\lz)^x-\gl_{-y}(z)^x)$$
$$=\f{\lD_{y,y}^{\ln/2}\,\lD_{z+\lz,y}^n}{\lD_{z,y}^{\ln+n}}\,\I_{D}d\lm(x)\,\f{\lD_{x,x}^\ln\,\lD_{\gl_{-y}(z+\lz),x}^n}{\lD_{\gl_{-y}(z),x}^{\ln+n}}\,(\lF_x\circ B_{x,x}B_{y,-x}^{-1}B_{y,y}^{1/2})((z+\lz)^{\gl_y(x)}-z^{\gl_y(x)}).$$
On the 'little' Hilbert space the action is
$$(\gl_y^{\la_\ln}\lF)_z=\f{\lD_{y,y}^{\ln/2}}{\lD_{z,y}^\ln}\,\lF_{\gl_{-y}(z)}\circ B_{y,y}^{1/2}B_{z,y}^{-1}.$$
Equivalently,
$$(\gl_y^{\la_\ln}\lF)_{\gl_y(x)}=\f{\lD_{x,-y}^\ln}{\lD_{y,y}^{\ln/2}}\,\lF_x\circ B_{x,-y}B_{y,y}^{-1/2}.$$
Since $\lm$ is $G$-invariant it follows that
$$(\IL_\ll\,\gl_y^{\la_\ln}\lF)_z(\lz)=\I_{D}d\lm(x')\,\f{\lD_{x',x'}^\ln\,\lD_{z+\lz,x'}^n}{\lD_{z,x'}^{\ln+n}}\,
((\gl_y^{\la_\ln}\lF)_{x'}\circ B_{x',x'})((z+\lz)^{x'}-z^{x'})$$
$$=\I_{D}d\lm(x)\,\f{\lD_{\gl_y(x),\gl_y(x)}^\ln\,\lD_{z+\lz,\gl_y(x)}^n}{\lD_{z,\gl_y(x)}^{\ln+n}}\,
((\gl_y^{\la_\ln}\lF)_{\gl_y(x)}\circ B_{\gl_y(x),\gl_y(x)})((z+\lz)^{\gl_y(x)}-z^{\gl_y(x)})$$
$$=\I_{D}d\lm(x)\,\f{\lD_{\gl_y(x),\gl_y(x)}^\ln\,\lD_{x,-y}^\ln\,\lD_{z+\lz,\gl_y(x)}^n}{\lD_{y,y}^{\ln/2}\,\lD_{z,\gl_y(x)}^{\ln+n}}
(\lF_x\circ B_{x,-y}B_{y,y}^{-1/2}B_{\gl_y(x),\gl_y(x)})((z+\lz)^{\gl_y(x)}-z^{\gl_y(x)}).$$
Thus, in order to verify the intertwining property, we must show that the terms involving the quasi-determinant and the terms involving the Bergman operators match. As a consequence of \er{10} we have
$$\lD_{\gl_y(u),\gl_y(v)}=\lD_{y,y}\lD_{u,-y}^{-1}\lD_{u,v}\lD_{y,-v}^{-1}.$$ 
As special cases we obtain
$$\lD_{\gl_y(x),\gl_y(x)}=\lD_{y,y}\lD_{x,-y}^{-1}\lD_{x,x}\lD_{y,-x}^{-1},$$
$$\lD_{z,\gl_y(x)}=\lD_{y,y}\lD_{\gl_{-y}(z),-y}^{-1}\lD_{\gl_{-y}(z),x}\lD_{y,-x}^{-1}=\lD_{z,y}\lD_{\gl_{-y}(z),x}\lD_{y,-x}^{-1},$$
$$\lD_{z+\lz,\gl_y(x)}=\lD_{y,y}\lD_{\gl_{-y}(z+\lz),-y}^{-1}\lD_{\gl_{-y}(z+\lz),x}\lD_{y,-x}^{-1}
=\lD_{z+\lz,y}\lD_{\gl_{-y}(z+\lz),x}\lD_{y,-x}^{-1}.$$
It follows that
$$\lD_{y,y}^{\ln/2}\,\lD_{z,y}^{-\ln-n}\,\lD_{z+\lz,y}^n\,\lD_{x,x}^\ln\,\lD_{\gl_{-y}(z),x}^{-\ln-n}\,\lD_{\gl_{-y}(z+\lz),x}^n
=\lD_{y,y}^{-\ln/2}\,\lD_{\gl_y(x),\gl_y(x)}^\ln\,\lD_{x,-y}^\ln\,\lD_{z,\gl_y(x)}^{-\ln-n}\,\lD_{z+\lz,\gl_y(x)}^n.$$
As a consequence of \er{9} we have 
$$B_{\gl_y(u),\gl_y(v)}=B_{y,y}^{1/2}B_{u,-y}^{-1}B_{u,v}B_{y,-v}^{-1}B_{y,y}^{1/2}$$ 
and hence $B_{x,x}B_{y,-x}^{-1}B_{y,y}^{1/2}=B_{x,-y}B_{y,y}^{-1/2}B_{\gl_y(x),\gl_y(x)}.$
\end{proof}
We will now show that the intertwiners $\IL_\ll$ can also be expressed in terms of {\bf differential} operators, which for the unit ball in $\Cl^d$ can be made very explicit. The following theorem is a far-reaching generalization of the 'jet construction' for the unit disk \cite{KM1}.
\begin{theorem}\label{f} For each $\lq\in\PL_\ll Z$ and $\lz\in Z$ the function
$$F_\lq(\lz,w):=\I_{D}d\lm(x)\,\lD_{x,x}^\ln\,e^{(x|w)}\,\lD_{\lz,x}^n\,\lq(B_{x,x}\lz^x)$$
is a (conjugate) polynomial in $w\in Z,$ and the associated constant coefficient holomorphic differential operator, in the variable 
$z,$ is the intertwining operator
$$(\IL_\ll(f\xt\lq))_z(\lz)=(F_\lq(\lz,\dl)f)(z).$$
Here we write $\lF\in H^2_\ln(D,\PL_\ll Z)$ as a (finite sum) of terms $\lF_z(\lz)=f(z)\lq(\lz),$ where $f\in\OL(D)$ and 
$\lq\in\PL_\ll Z.$ 
\end{theorem}
\begin{proof} The Lie algebra $\gL=aut(D)$ of $G$ consists of all completely integrable holomorphic vector fields $\lg:D\to Z,$ realized as holomorphic differential operators
$$(\lg^\dl\,f)(z):=(\dl_zf)\,\lg(z)$$
on holomorphic functions $f:D\to\Cl.$ Then the commutation relation
$$\[\lx^\dl,\lh^\dl\]=(\lx^\dl\,\lh-\lh^\dl\,\lx)^\dl$$
holds. For $\lg\in\gL^\Cl$ consider the {\bf infinitesimal actions} associated with \er{19} and \er{24}, resp. The intertwining property 
$$\lg^{\d\lb_\ln}\,\IL_\ll=\IL_\ll\,\lg^{\d\la_\ln}$$ 
is valid on the dense subspace of smooth vectors in $H^2_\ln(D,\PL_\ll Z).$ On the 'little' Hilbert space the infinitesimal action, as a function of $(z,\lz),$ has the form
$$-\lg^{\d\la_\ln}\lF=\f{\ln}p\,tr(\dl_z\lg)\cdot\lF+(\lg(z))^\dl\,\lF+((\dl_z\lg)\lz)^\ld\lF.$$
Here $()^\dl$ denotes differentiation in the $z$-variable and $()^\ld$ denotes differentiation in the $\lz$-variable. On the 'big' Hilbert space the infinitesimal action, as a function of $(z,\lz),$ has the form
$$-\lg^{\d\lb_\ln}\lF=(\f{\ln+n}p\,tr(\dl_z\lg)-\f np\,tr(\dl_{z+\lz}\lg)\cdot\lF+(\lg(z))^\dl\,\lQ+(\lg(z+\lz)-\lg(z))^\ld\lF.$$
There is a Cartan decomposition $\c{\gL}=\kL\oplus\pL,$ where $\kL$ is the Lie algebra of all Jordan triple derivations (identified as linear vector fields) and $\pL$ consists of all {\bf infinitesimal transvections}
$$\c v=(v-Q_zv)^\dl$$
for $v\in Z$. We have $\dl_z\c v=-2z\Box v^*$ and hence
$$tr(\dl_z\c v)=-2\,tr(z\Box v^*)=-p(z|v).$$  
It follows that
$$\c v^{\d\la_\ln}=\ln(z|v)+(Q_zv)^\dl-v^\dl+2\{zv^*\lz\}^\ld$$  
on the little Hilbert space. Since
$$\c v_{z+\lz}-\c v_z=Q_zv-Q_{z+\lz}v=-2\{zv^*\lz\}-Q_\lz v$$
we obtain similarly
$$\c v^{\d\lb_\ln}=\ln(z|v)-n(\lz|v)+(Q_zv)^\dl-v^\dl+2\{zv^*\lz\}^\ld+(Q_\lz v)^\ld$$
for the big Hilbert space. Here $(z|v)$ and $(\lz|v)$ denote the multiplication operators by $(z|v)$ and $(\lz|v),$ resp. In both cases the infinitesimal action has the same $\Cl$-linear part, namely $-v^\dl.$ It follows that
$$v^\dl\IL_\ll(f\xt\lq)=\IL_\ll((v^\dl f)\xt\lq)$$
for all $f\in\PL Z$ and all $v\in Z.$ Thus 
$$(\IL_\ll(f\xt\lq))_z(\lz)=(F_\lq(\lz,\dl)f)(z)$$ 
is a constant coefficient holomorphic differential operator in the $z$-variable. Its symbol $F_\lq(\lz,w),$ as a (conjugate) polynomial in $w\in Z,$ is determined by the property
$$F_\lq(\lz,w)\,e^{(z|w)}=(F_\lq(\lz,\dl)e^{(-|w)})(z)=(\IL_\ll(e^{(-|w)}\xt\lq))_z(\lz)
=\I_{D}d\lm(x)\,\f{\lD_{x,x}^\ln\,\lD_{z+\lz,x}^n}{\lD_{z,x}^{\ln+n}}\,e^{(x|w)}\,\lq(B_{x,x}((z+\lz)^x-z^x)).$$
Putting $z=0$ we obtain
$$F_\lq(\lz,w)=\I_{D}d\lm(x)\,\lD_{x,x}^\ln\,e^{(x|w)}\,\lD_{\lz,x}^n\,\lq(B_{x,x}\lz^x).$$
\end{proof}
We will now specialize to the rank $1$ case of the unit ball $D$ in $Z=\Cl^d.$ Define
$$\lz^\dl=(\lz|\o\dl):=\S_{i=1}^d(\lz|u_i)u_i^\dl,$$
where $u_i$ is any orthonormal basis of $Z.$ 
\begin{theorem} For the unit ball $D$ and each $0\le\ll\le n$ the operator
$$\IL_\ll=\S_{k=0}^{n-\ll}\f{(-1)^k}{(2\ll+\ln)_k}\b{n-\ll}{k}(\lz|\o\dl)^k
=\S_{k=0}^{n-\ll}\f{(\ll-n)_k}{(2\ll+\ln)_k}\f{(\lz|\o\dl)^k}{k!}=\S_{k=0}^{n-\ll}\f{(\ll-n)_k}{(2\ll+\ln)_k}\EL^k(\lz,\o\dl)$$
is an intertwiner $H^2_\ln(D,\PL_\ll Z)\to H^2_\ln(D,\PL^nZ).$
\end{theorem}
\begin{proof} Consider the probability measure
$$d\lm_\ln(x)=c_\ln\,(1-(x|x))^{\ln-d-1}\,dx$$
on $D,$ for parameter $\ln>d=p-1.$ Then
$$\I_D d\lm_{\ln+\ll}(w)\,e^{(w|b)}\,(\lz|w)^{k+\el}(w|a)^\el=\I_D d\lm_{\ln+\ll}(w)\,\f{(w|b)^k}{k!}\,(\lz|w)^{k+\el}(w|a)^\el$$
$$=\f1{(\ln+\ll)_{k+\el}}\((w|\lz)^{k+\el}\|\f{(w|b)^k}{k!}(w|a)^\el\)_F 
=\f1{(\ln+\ll)_{k+\el}}\((a^\dl)^\el(w|\lz)^{k+\el}\|\f{(w|b)^k}{k!}\)_F$$
$$=\f{(k+1)_\el}{(\ln+\ll)_{k+\el}}(\lz|a)^\el\((w|\lz)^k\|\f{(w|b)^k}{k!}\)_F=\f{(k+1)_\el}{(\ln+\ll)_{k+\el}}(\lz|a)^\el(\lz|b)^k.$$
The Jordan triple product on $Z=\Cl^d$ is given by
$$2\{uv^*w\}=(u|v)w+(w|v)u.$$
Realizing $\Cl^d=\Cl^{1\xx d}$ via row vectors, the Bergman operator is
$B_{w,w}\lz=(1-ww^*)\lz(1-w^*w)=(1-(w|w))(\lz-(\lz|w)w).$ The quasi-inverse is
$$\lz^w=(1-\lz w^*)^{-1}\lz=\f{\lz}{1-(\lz|w)}.$$
It follows that
$$B_{w,w}\lz^w=\f{1-(w|w)}{1-(\lz|w)}(\lz-(\lz|w)w).$$
Now let $\lq$ be a $\ll$-homogeneous polynomial. Then
$$\lD_{\lz,w}^n\,\lq(B_{w,w}\lz^w)=(1-(\lz|w))^n\f{(1-(w|w))^\ll}{(1-(\lz|w))^\ll}\lq(\lz-(\lz|w)w)
=(1-(w|w))^\ll\,(1-(\lz|w))^{n-\ll}\lq(\lz-(\lz|w)w)$$
$$=\S_{k=0}^{n-\ll}(-1)^k\b{n-\ll}{k}(\lz|w)^k(1-(w|w))^\ll\,\lq(\lz-(\lz|w)w).$$
Therefore we have, for some (known) constant $c'$ involving the different normalization for $\ln$ and $\ln+\ll,$
$$c'\I_D d\lm_\ln(w)\,e^{(w|b)}\,\lD_{\lz,w}^n\,\lq(B_{w,w}\lz^w)
=\S_{k=0}^{n-\ll}(-1)^k\b{n-\ll}{k}\I_D d\lm_{\ln+\ll}(w)\,e^{(w|b)}\,(\lz|w)^k\,\lq(\lz-(\lz|w)w)$$
$$=c\S_{k=0}^{n-\ll}\f{(-1)^k}{(2\ll+\ln)_k}\b{n-\ll}{k}(\lz|b)^k\lq(\lz),$$
where $c=\f{\ll+\ln-1}{2\ll+\ln-1}$ is independent of $k,$ and the last step follows from the calculation
$$\I_D d\lm_{\ln+\ll}(w)\,e^{(w|b)}\,(\lz|w)^k\,(\lz-(\lz|w)w|a)^\ll
=\S_{\el=0}^\ll(-1)^\el\b{\ll}{\el}(\lz|a)^{\ll-\el}\I_D d\lm_{\ln+\ll}(w)\,e^{(w|b)}\,(\lz|w)^{k+\el}(w|a)^\el$$
$$=\S_{\el=0}^\ll(-1)^\el\b{\ll}{\el}(\lz|a)^{\ll-\el}\f{(k+1)_\el}{(\ln+\ll)_{k+\el}}(\lz|a)^\el(\lz|b)^k
=\f{c}{(2\ll+\ln)_k}(\lz|a)^\ll(\lz|b)^k,$$
based on the formula
$$\S_{\el=0}^\ll(-1)^\el\b{\ll}{\el}\f{\lG(a+\el)}{\lG(b+\el)}=\f{B(\ll+b-a,a)}{\lG(b-a)}.$$
On the other hand
$$(\lz|\o\dl)\,e^{(z|b)}\,\lq(\lz)=(\lz|u_i)u_i^\dl e^{(z|b)}\,\lq(\lz)=(\lz|u_i)(u_i|b)e^{(z|b)}\,\lq(\lz)=(\lz|b)e^{(z|b)}\,\lq(\lz)$$
and therefore $(\lz|\o\dl)^k\,e^{(z|b)}\,\lq(\lz)=(\lz|b)^k\,e^{(z|b)}\,\lq(\lz).$
\end{proof}
It would be of interest to obtain explicit formulas for the intertwining differential operators, e.g. in the case of the Lie balls of rank $2.$

\section{Reproducing kernels}
We now consider the vector-valued Bergman spaces $H^2_\ln(D,V),$ which also depend on the scalar parameter $\ln,$ and their reproducing matrix kernel functions, denoted by $\KL_{z,w}^\ln.$ The following two properties were already mentioned in Section 1. 
\begin{lemma} In general, we have the covariance property
$$g^{\la_\ln}\,\KL_w^\ln=\o{\det(\dl_wg)}^{\ln/p}\,\KL_{gw}^\ln([g]_w^*)^{-1}$$
for the reproducing kernel of $\KL^\ln$ of $H^2_\ln(D,V).$
\end{lemma}
\begin{proof} Let $\lF\in H^2_\ln(D,V)$ and $\lq\in V.$ Then
$$(\lF|g^{\la_\ln}\,\KL_w^\ln\,\lq)=(g^{-\la_\ln}\lF|\KL_w^\ln\,\lq)=((g^{-\la_\ln}\lF)_w|\lq)_V=(\det(\dl_wg)^{\ln/p}[g]_w^{-1}\lF_{gw}|\lq)_V$$
$$=\o{\det g^\1(w)}^{\ln/p}(\lF_{gw}|([g]_w^*)^{-1}\lq)_V=\o{\det(\dl_wg)}^{\ln/p}(\lF|\KL_{gw}^\ln([g]_w^*)^{-1}\lq).$$
\end{proof}
\begin{corollary} 
\be{36}\KL_{z,w}^\ln=\det(\dl_zg)^{\ln/p}\,\o{\det(\dl_wg)^{\ln/p}}\,[g]_z^{-1}\,\KL_{gz,gw}^\ln\,([g]_w^*)^{-1}.\ee
\end{corollary}
\begin{proof} Since $g^{\la_\ln}$ is unitary, it follows that
$$\KL_{z,w}^\ln=(\KL_z^\ln)^*\KL_w^\ln=(g^{\la_\ln}\KL_z^\ln)^*(g^\la\KL_w^\ln)=\det(\dl_zg)^{\ln/p}\,(\KL_{gz}^\ln([g]_z^*)^{-1})^*\o{\det(\dl_wg)}^{\ln/p}\,\KL_{gw}^\ln([g]_w^*)^{-1}$$
$$=\det(\dl_zg)^{\ln/p}\,\o{\det(\dl_wg)^{\ln/p}}\,[g]_z^{-1}\,(\KL_{gz}^\ln)^*\KL_{gw}^\ln\,([g]_w^*)^{-1}
=\det(\dl_zg)^{\ln/p}\,\o{\det(\dl_wg)^{\ln/p}}\,[g]_z^{-1}\,\KL_{gz,gw}^\ln\,([g]_w^*)^{-1}.$$
\end{proof}
As a consequence we obtain
\begin{lemma}
\be{}\KL_{x,x}^\ln=\lD_{x,x}^{-\ln}\,[\gl_x]_0\,\KL_{0,0}^\ln\,[\gl_x]_0^*.\ee
\end{lemma}
\begin{proof} Apply \er{36} to $g=\gl_x$ and $z=w=0.$
\end{proof}
We may write
\be{38}\KL_{z,w}^\ln=\lD_{z,w}^{-\ln}\,\JL_{z,w}^\ln,\ee
where $\JL_{z,w}^\ln$ is obtained by polarizing
$$\JL_{x,x}^\ln=[\gl_x]_0\,\KL_{0,0}^\ln\,[\gl_x]_0^*.$$
In general this is not a positive definite kernel. 

For the 'little' Hilbert space $H^2_\ln(D,\PL_\ll Z),$ the reproducing kernel is essentially unique, and can be expressed in closed form:
\begin{proposition} For each partition $\ll$ the reproducing kernel $\KL^{\ln,\ll}_w:\PL_\ll Z\to H^2_\ln(D,\PL_\ll Z)$ is (after normalization) given by
\be{31}(\KL^{\ln,\ll}_w\lq)_z(\lz)=\lD_{z,w}^{-\ln}\,\lq(B_{z,w}^{-1}\lz)=\lD_{z,w}^{-\ln}\,(\JL^\ll_{z,w}\lq)(\lz)\qquad\forall\,\lq\in\PL_\ll Z.\ee Here
$$(\JL^\ll_{z,w}\lq)(\lz)=(B_{z,w}^\lp\lq)(\lz)=\lq(B_{z,w}^{-1}\lz)$$
is independent of $\ln.$
\end{proposition}
\begin{proof} The general covariance formula \er{62} yields
$$\KL_{x,x}^{\ln,\ll}=\lD_{x,x}^{-\ln}(\dl_0\gl_x)(\dl_0\gl_x)^*.$$
Now the assertion follows with \er{18}.
\end{proof}
\begin{proposition} On reproducing kernel functions the intertwiner \er{32} has the value
$$(\IL_\ll(\KL^{\ln,\ll}_w\lq))_z(\lz)=\f{\lD_{z+\lz,w}^n}{\lD_{z,w}^{\ln+n}}\,\lq((z+\lz)^w-z^w)$$
for all $\lq\in\PL_\ll Z.$
\end{proposition}
\begin{proof} The reproducing property of \er{31} yields
$$(\lF_w|\lq)_F=(\lF|\KL^{\ln,\ll}_w\lq)=\I_D d\lm(x)\,\lD_{x,x}^\ln\(\lF_x\|(\KL^{\ln,\ll}_w q)_x\circ B_{x,x}\)_F =\I_D d\lm(x)\,\f{\lD_{x,x}^\ln}{\lD_{x,w}^\ln}\(\lF_x\|q\circ B_{x,w}^{-1}B_{x,x}\)_F$$
for all $\lF\in H^2_\ln(D,\PL_\ll Z).$ Now consider the holomorphic map
$$D\ni x\mapsto\lF_x:=\lD_{x,z}^{-\ln-n}\,\lD_{x,z+\lz}^n\,\EL^\ll_{(z+\lz)^x-z^x}\in\PL_\ll Z.$$
Then
$$(\IL_\ll(\KL^{\ln,\ll}_w\lq))_z(\lz)
=\I_{D}d\lm(x)\,\f{\lD_{x,x}^\ln\,\lD_{z+\lz,x}^n}{\lD_{z,x}^{\ln+n}}\,((\KL^\ln_w\lq)_x\circ B_{x,x})((z+\lz)^x-z^x)$$
$$=\I_{D}d\lm(x)\,\f{\lD_{x,x}^\ln\,\lD_{z+\lz,x}^n}{\lD_{z,x}^{\ln+n}\,\lD_{x,w}^\ln}\,(\lq\circ B_{x,w}^{-1}B_{x,x})((z+\lz)^x-z^x)$$
$$=\I_{D}d\lm(x)\,\f{\lD_{x,x}^\ln\,\lD_{z+\lz,x}^n}{\lD_{z,x}^{\ln+n}\,\lD_{x,w}^\ln}\(\EL^\ll_{(z+\lz)^x-z^x}\|\lq\circ B_{x,w}^{-1}B_{x,x}\)_F$$
$$=\I_{D}d\lm(x)\,\f{\lD_{x,x}^\ln}{\lD_{x,w}^\ln}\(\lF_x\|\lq\circ B_{x,w}^{-1}B_{x,x}\)_F=(\lF_w|\lq)_F$$
$$=\f{\lD_{z+\lz,w}^n}{\lD_{z,w}^{\ln+n}}\,(\EL^\ll_{(z+\lz)^w-z^w}|\lq)_F=\f{\lD_{z+\lz,w}^n}{\lD_{z,w}^{\ln+n}}\,\lq((z+\lz)^w-z^w)$$
using the reproducing property of $\EL^\ll$ with respect to the Fischer-Fock inner product.
\end{proof}
We may describe $\KL^{\ln,\ll}_{z,w}$ in terms of its integral kernel
$$\KL^{\ln,\ll}_{z,w}(\lz,\lo):=(\EL^\ll_\lz|\KL^{\ln,\ll}_{z,w}\EL^\ll_\lo)_F=(\KL^{\ln,\ll}_{z,w}\EL^\ll_\lo)(\lz)
=\lD_{z,w}^{-\ln}\,\EL^\ll_\lo(B_{z,w}^{-1}\lz)=\lD_{z,w}^{-\ln}\,\EL^\ll(B_{z,w}^{-1}\lz,\lo)$$
with respect to Gauss measure. Using an orthonormal basis 
\be{53}\lF^\la_z(\lz)=\S_i p^{\la,i}(z)\,q^{\la,i}(\lz)\ee 
of $H^2_\ln(D,\PL_\ll Z)$ we have 
$\KL^{\ln,\ll}_{z,w}=\S_\la\lF^\la_z(\lF^\la_w)^*\in End(\PL_\ll Z)$ and hence
\be{33}\KL^{\ln,\ll}_{z,w}(\lz,\lo)=\S_\la\lF^\la_z(\lz)\o{\lF^\la_w(\lo)}.\ee
For the 'big' Hilbert spaces $H^2_\ln(D,\PL^nZ)$ the $G$-action is not irreducible and covariant kernel functions are not unique. All possible $G$-covariant kernel functions are obtained as follows: Choose a tuple
$$\cl=(c_\ll)_{\ll\le n}$$
of constants $c_\ll>0.$ Then there exists a unique $K$-invariant inner product on $\PL^nZ=:\PL_\cl^nZ$ such that the resulting Hilbert space $H^2_\ln(D,\PL_\cl^nZ)$ has the orthonormal basis $\{\F c_\ll\IL_\ll\lF^\la\},$ using the intertwiners $\IL_\ll$ and the orthonormal basis \er{53}. The reproducing kernel $\KL^{\ln,\cl}_{z,w}\in End(\PL_\cl^nZ)$ is determined by
\be{35}\KL^{\ln,\cl}_{z,w}(\lz,\lo):=\S_\ll c_\ll\S_\la(\IL_\ll\lF^\la)_z(\lz)\o{(\IL_\ll\lF^\la)_w(\lo)}.\ee
\begin{lemma}\label{e} Every $K$-invariant sesqui-polynomial $F$ on $Z\xx\o Z$ has the form
$$F(\lz,\lo)=\S_\lm a_\lm\,\EL^\lm(\lz,\lo)$$
over partitions $\lm,$ for suitable constants $a_\lm.$
\end{lemma}
\begin{proof} In terms of orthonomal bases $\lf_\lm^m$ of $\PL_\lm Z$ we may write
$$F(\lz,\lo)=\S_{\lm,\lk}\S_{m,k}a_{\lm,\lk}^{m,k}\lf_\lm^m(\lz)\o{\lf_\lk^k(\lo)}$$ 
Since $F$ is $K$-invariant, the operator 
$$F^\#:=\S_{\lm,\lk}\S_{m,k}a^{m,k}_{\lm,\lk}\lf^m_\lm(\lf^k_\lk)^*$$
acting on $\PL Z$ commutes with the action of $K.$ Since this action is multiplicity-free it follows that $F^\#$ is a block-diagonal operator. This implies $a_{\lm,\lk}^{m,k}=0$ whenever $\lm\ne\lk.$ Moreover, on each irreducible subspace $\PL_\lm Z,$ the operator 
$F^\#$ acts as a multiple of the identity. Therefore $a_{\lm,\lm}^{m,k}=a_\lm$ is independent of $m$ and $k.$
\end{proof}
\begin{theorem}\label{g} The integral kernel $\KL^{\ln,\cl}_{0,0}$ has the form
$$\KL^{\ln,\cl}_{0,0}(\lz,\lo)=\S_{\ll\le n}c_\ll\S_{\ll\le\lm\le n}a_\lm^\ll(\ln)\,\EL^\lm(\lz,\lo),$$
where 
\be{37}A(\ln)=(a_\lm^\ll(\ln))_{\ll\le\lm}\ee 
is a lower triangular matrix of functions which are continuous in the parameter $\ln.$ Moreover, $a_\ll^\ll=1.$ 
\end{theorem}
\begin{proof} For any $v\in Z$ the element $\o\tL_w=\exp(Q_\lz w)\in G^\Cl$ acts on $\PL^nZ$ via
$$(\o\tL_w\lq)(\lz)=\det(\dl_\lz\o\tL_{-w})^{-n/p}\,\lq(\o\tL_{-w}(\lz))=(\det B_{\lz,w}^{-1})^{-n/p}\,\lq(\lz^w)=\lD_{\lz,w}^n\,\lq(\lz^w).$$
This implies
\be{34}(\o\tL_w\lq)(\lz)=\S_{\lm\le n}\S_m\lf_\lm^m(\lz)(\lf_\lm^m|\o\tL_w\lq)_F.\ee
On the other hand, we have for every $\lg\in\gL^\Cl$
$$f(\exp(-\lg)(\lz))=\S_{k\ge 0}((\lg^\ld)^k\,f)(\lz)$$
for any holomorphic function $f(\lz).$ Here $(\lg^\ld)^k$ are the differential operator powers of the vector field $\lg(\lz)^\ld.$ Now let $\lq\in\PL_\ll Z.$ For $\lg(\lz)=Q_\lz w$ it follows from \cite{U1} that the polynomial 
$(Q_\lz w)^\ld\lq$ has only components of type $\ll+\Le_i$ for some $1\le i\le r.$ It follows that for each fixed $k\ge 0$ the polynomial $(((Q_\lz w)^\ld)^k\,\lq)(\lz)$ has only components of type $\lm\ge\ll.$ The same is true after multiplying with the 
$\lz$-polynomial $\lD_{\lz,w}^n.$ Since the whole series
$$\lD_{\lz,w}^n\,\lq(\lz^w)=\lD_{\lz,w}^n\S_{k\ge 0}(((Q_\lz w)^\ld)^k\,\lq)(\lz)$$
belongs to $\PL^nZ$ we can sharpen \er{34} to
$$(\o\tL_w\lq)(\lz)=\S_{\ll\le\lm\le n}\S_m\lf_\lm^m(\lz)(\lf_\lm^m|\o\tL_w\lq)_F.$$
whenever $\lq\in\PL_\ll Z.$ Setting $z=0$ in \er{32} we obtain
$$\IL_\ll(f\xt\lq)_0(\lz)=\I_D d\lm(x)\,\lD_{x,x}^\ln\,\lD_{\lz,x}^n\,f(x)\,\lq(B_{x,x}\lz^x)
=\S_{\ll\le\lm\le n}\S_m\lf_\lm^m(\lz)\I_D d\lm(x)\,\lD_{x,x}^\ln\,f(x)\,(\lf_\lm^m|\o\tL_x(\lq\circ B_{x,x}))_F.$$
Specializing \er{35} to $z=w=0$ we obtain
$$\KL^{\ln,\cl}_{0,0}(\lz,\lo)=\S_\ll c_\ll\S_\la(\IL_\ll\lF^\la)_0(\lz)\o{(\IL_\ll\lF^\la)_0(\lo)}$$
$$=\S_\ll c_\ll\S_\la\S_{\ll\le\lm,\lk\le n}\S_{m,k}\lf_\lm^m(\lz)\o{\lf_\lk^k(\lo)}\I_D d\lm(x)\,\lD_{x,x}^\ln\,(\lf_\lm^m|\o\tL_x(\lF^\la_x\circ B_{x,x}))_F\I_D d\lm(y)\,\lD_{y,y}^\ln\,(\o\tL_y(\lF^\la_y\circ B_{y,y})|\lf_\lk^k)_F$$
$$=\S_\ll c_\ll\S_{\ll\le\lm,\lk\le n}\S_{m,k}\,a_{\lm,\lk}^{m,k}(\ln)\lf_\lm^m(\lz)\o{\lf_\lk^k(\lo)},$$
where
$$a_{\lm,\lk}^{m,k}(\ln)=\S_\la\I_D d\lm(x)\,\lD_{x,x}^\ln\,(\lf_\lm^m|\o\tL_x(\lF^\la_x\circ B_{x,x}))_F
\I_D d\lm(y)\,\lD_{y,y}^\ln\,(\o\tL_y(\lF^\la_y\circ B_{y,y})|\lf_\lk^k)_F$$
is continuous in $\ln.$ Since the sesqui-polynomial $\KL^{\ln,\cl}_{0,0}$ is $K$-invariant, Lemma \er{e} implies that the non-zero terms occur only for $\lm=\lk$ and $a_{\lm,\lm}^{m,k}=a_\lm$ is independent of $m$ and $k.$
\end{proof}
For the unit ball in $\Cl^d,$ we can describe $\KL^{\ln,\cl}_{0,0}$ explicitly.
\begin{theorem} For the unit ball in $\Cl^d,$ the general $G$-covariant kernel has the form
$$\KL^{\ln,\cl}_{0,0}(\lz,\lo)=\S_{\el=0}^n(\lz|\lo)^\el\S_{\ll=0}^\el\lk^\ll_\el(\ln)\,\f{c_\ll}{\ll!},$$
where $c_\ll>0$ are arbitrary constants, and $(\lk^\ll_\el)$ is a lower triangular matrix with universal coefficients 
$$\lk^\ll_\el(\ln)=\(\f{(\ll-n)_{\el-\ll}}{(\ln+2\ll)_{\el-\ll}}\)^2\S_{p+q=\el-\ll}\b{\ln+\ll-1}p\binom{\ll}q.$$
\end{theorem}
\begin{proof} The explicit form $\IL_\ll=\S_{k}\lm^\ll_k\,\EL_k(\lz,\o\dl)$ obtained in Theorem \er{f} yields
$$\KL^{\ln,\cl}_{z,w}(\lz,\lo)=\S_\ll c_\ll\S_{k,h}\lm^\ll_k\lm^\ll_h\S_\la(\EL_k(\lz,\o\dl)\lF^\la)_z\o{(\EL_h(\lo,\o\dl)\lF^\la)_w}.$$
For $p\in\PL Z$ and $q\in\PL_\ll Z$ we have
$$\EL_k(\lz,\o\dl)\,p(z)q(\lz)|_{z=0}=p_k(\lz)q(\lz),$$
where $p_k\in\PL_k Z$ is the $k$-th component of $p.$ Using the decomposition \er{53} it follows that
\be{54}\KL^{\ln,\cl}_{0,0}(\lz,\lo)
=\S_\ll c_\ll\S_{k,h}\lm^\ll_k\lm^\ll_h\S_{\la,i,j}p_k^{\la,i}(\lz)q^{\la,i}(\lz)\o{p_h^{\la,j}(\lo)q^{\la,j}(\lo)}
=\S_\ll c_\ll\S_{k,h}\lm^\ll_k\lm^\ll_h\KL^{\ln,\ll}_{k,h}(\lz,\lo),\ee
where $\KL_{k,h}^{\ln,\ll}(\lz,\lo)$ denotes the component of degree $k$ in $z$ and degree $h$ in $w,$ resp. Since
$$B_{z,w}^{-1}\lz=B_{z^w,-w}\lz=(1+(z^w|w))(\lz+(\lz|w)z^w)=\(1+\f{(z|w)}{1-(z|w)}\)\(\lz+\f{(\lz|w)z}{1-(z|w)}\)
=\f1{1-(z|w)}\(\lz+\f{(\lz|w)z}{1-(z|w)}\)$$
it follows from \er{31} that
$$\KL^{\ln,\ll}_{z,w}(\lz,\lo)=(1-(z|w))^{-\ln}\,\f1{\ll!}(B_{z,w}^{-1}\lz|\lo)^\ll
=(1-(z|w))^{-\ln-\ll}\,\f1{\ll!}\((\lz|\lo)+\f{(\lz|w)(z|\lo)}{1-(z|w)}\)^\ll$$
$$=\f1{\ll!}\S_{q=0}^\ll\b{\ll}q(\lz|\lo)^{\ll-q}(\lz|w)^q(z|\lo)^q(1-(z|w))^{-\ln-\ll-q}$$
$$=\f1{\ll!}\S_{q=0}^\ll\b{\ll}q(\lz|\lo)^{\ll-q}(\lz|w)^q(z|\lo)^q\S_{p\ge 0}\b{\ln+\ll+q+p-1}p(z|w)^p.$$
Choosing the components of degree $k$ in $z$ and degree $h$ in $w$ we obtain non-zero terms for $k=h$ and
$$\KL^{\ln,\ll}_{k,k}(\lz,\lo)|_{0,0}=\f1{\ll!}(\lz|\lo)^{\ll+k}\S_{p+q=k}\b{\ln+\ll+k-1}p\b{\ll}q.$$
In view of \er{54} this implies
$$\KL^{\ln,\cl}_{0,0}(\lz,\lo)=\S_{\ll=0}^n\,c_\ll\S_{k=0}^{n-\ll}(\lm^\ll_k)^2\,\KL^{\ln,\ll}_{k,k}(\lz,\lo)$$
$$=\S_{\ll=0}^n\,\f{c_\ll}{\ll!}\S_{k=0}^{n-\ll}\(\f{(\ll-n)_k}{(\ln+2\ll)_k}\)^2\,(\lz|\lo)^{\ll+k}\S_{p+q=k}\b{\ln+\ll+k-1}p\b{\ll}q$$
$$=\S_{\el=0}^n(\lz|\lo)^\el\S_{\ll=0}^\el\(\f{(\ll-n)_{\el-\ll}}{(\ln+2\ll)_{\el-\ll}}\)^2
\S_{p+q=\el-\ll}\b{\ln+\el-1}p\b{\ll}q\,\f{c_\ll}{\ll!}$$
where in the last step we have put $\el=\ll+k.$
\end{proof}

\section{Boundedness of multiplication operators}
We will now apply the structure of reproducing kernels to study boundedness of multiplication operators
$$(f^\oc\lF)_z(\lz)=f(z)\lF_z(\lz)$$  
for bounded holomorphic functions $f\in\OL(D).$ The following Lemma is well-known. 
\begin{lemma}
Let $\HL$ be a Hilbert space consisting of holomorphic maps on some domain $D\ic\Cl^d$ taking values in $V=\Cl^k.$ Assume that it possesses a reproducing kernel $\KL.$ Then the multiplication by the coordinate functions is bounded on $H$ if and only if there exists a real number $b$ such that 
$$\LL_{z,w}:=(b^2-(z|w))\KL_{z,w}\qquad\forall\,z,w\in D$$
is positive (semi-definite) on $D,$ with values in $End(V).$
\end{lemma}
\begin{proof} A dense subset of $\HL$ is obtained by taking finite sums
$$\lF=\S_h\KL_{z_h}\lq_h$$
of kernel vectors, with $z_h\in D$ and $\lq_h\in V.$ The eigenvalue equation
$$\o z_i^\oc\KL_z\lq=\o z^i\,\KL_z\lq$$
implies that $T:=\S_{i=1}^d z_i^\oc\,\o z_i^\oc$ satisfies
$$(\lF|T\lF)=\S_{i=1}^d\S_{h,k}(\o z_i^\oc\KL_{z_h}\lq_h|\o z_i^\oc\KL_{z_k}\lq_k)=\S_{i=1}^d\S_{h,k}z_h^i\,\o z_k^i(\KL_{z_h}\lq_h|\KL_{z_k}\lq_k)=\S_{h,k}(z_h|z_k)(\lq_h|\KL_{z_h,z_k}\lq_k)$$
and, similarly,
$$(\lF|(b^2-T)\lF)=\S_{h,k}(b^2-(z_h|z_k))(\KL_{z_h}\lq_h|\KL_{z_k}\lq_k)=\S_{h,k}(\lq_h|\LL_{z_h,z_k}\lq_k).$$
This proves that the boundedness of the multiplication by the coordinate functions implies positivity of the kernel $\LL.$
The proof in the other direction follows exactly as in \cite[Section 3.2]{BM}
\end{proof}
As an application we consider the classical case of the Hilbert space $H_\ln$ of holomorphic functions on the unit ball $D=\Bl_d\ic\Cl^d,$ with reproducing kernel $\lD_{z,w}^{-\ln}=(1-(z|w))^{-\ln},$ for $\ln>0.$ The Bergman space corresponds to $\ln=d+1.$ The binomial expansion shows that $\lD^{-\ln}_{z,w}$ is a positive definite kernel on $\Bl_d$ for all $\ln> 0.$ The monomials form an orthogonal basis in the  Hilbert space $H^2_\ln(D,\Cl)$ determined by this positive definite kernel and for the norm we have
$$\|z^I\|_\ln^{-1}=(-1)^{|I|}\b{-\ln}{|I|}\b{|I|}{i_1\,\cdots,i_d},$$
where $I=(i_1,\ldots,i_d)$ and $|I|=i_1+\cdots+i_d.$ The multiplication by the coordinate functions $z_1^\oc,\ldots,z_d^\oc$ on 
$H^2_\ln(D,\Cl)$ are bounded. For $\ln\ge d,$ the inner product of $H^2_\ln(D,\Cl)$ is induced by a measure supported on the closed unit ball (if $\ln=d,$ the support is the boundary of the ball). In this case, the boundedness is evident. In general, the boundedness is established by showing that the operator $z_1^\oc\o z_1^\oc+\cdots+z_d^\oc\o z_d^\oc$ is the diagonal operator $\f{|I|+1}{\nu+|I|}{\rm Id},$ where ${\rm Id}$ is the identity operator on the monomials of degree $|I|.$  Clearly, this diagonal operator on 
$H^2_\ln(D,\Cl)$ is bounded for all $\ln>0.$ However, a different argument using the above Lemma will be useful in the more general context. Thus to prove that the multiplication operators $z_1^\oc,\ldots,z_d^\oc$ are bounded on $H^2_\ln(D,\Cl),$ it is enough to show that $(b^2-(z|w))\lD^{-\nu}_{z,w}$ is a positive kernel for some $b\in\Rl.$ Now we have
$$(b^2-(z|w))\lD^{-\ln}_{z,w}=b^2+\S_{k=0}^\oo(-1)^{k+1}\(\b{-\ln}{k+1}+\b{-\ln}{k}\)(z|w)^{k+1}
=b^2+\S_{k=0}^\oo\f{(\ln)_k}{(k+1)!}\(b^2-\f{k+1}{\ln+k}\)(z|w)^{k+1}.$$
Therefore, if $\ln\ge 1,$ choosing $b=1$ will ensure the positivity of $(b^2-(z|w))\lD^{-\nu}_{z,w}.$ On the other hand, for $0<\ln<1,$ choosing $b=\f1\ln$ will do. \\

In the general setting, let $\lp$ be an irreducible holomorphic representation of $K^\Cl,$ with highest weight $\lL_0.$ The {\bf Wallach set} $W_\lp,$ consisting of all parameters $\ln\ge 0$ such that the corresponding (scalar or matrix valued) kernel $\KL_{z,w}^\ln$ is positive, is a union of an open half-line $\{\ln:\,\ln>\ln_\lp\}$ (continuous part) and a finite subset of $[0,\ln_\lp]$ (discrete part) \cite{W}. For the scalar-valued Bergman spaces, corresponding to the trivial representation, the Faraut-Koranyi formula \er{23} implies
$$W_0=\{i\f a2:\,0\le i\le r-1\}\iu\{\ln:\,\ln>\f a2(r-1)\}.$$
Thus the cut-off point in the scalar case is
$$\ln_0=\f a2(r-1).$$
In particular, $\ln_0=0$ for the unit ball. Now consider the kernel functions $\KL_{z,w}^{\ln,\ll}$ associated with a partition 
$\ll.$ Let $W_\ll$ denote the corresponding Wallach set, with cutoff parameter $\ln_\ll.$ The empty partition $\ll=0$ corresponds to the scalar case. By \cite[Theorem 2.1]{FK} 
the representation $\PL_\ll Z$ has the highest weight
$$\lL_0=-\S_{i=1}^r \ll_i\lg_i$$
with respect to the Harish-Chandra strongly orthogonal roots $\lg_1<\ldots<\lg_r.$ Consider the character
$$\lL_1(u\Box v^*)=\f12(u|v)=\f1p tr(u\Box v^*)$$ on $\kL^\Cl,$ normalized by
$$\lL_1 H_{\lg_i}=2\lL_1(e_i\Box e_i^*)=(e_i|e_i)=1\qquad\forall\,1\le i\le r.$$
Then we have
$$(e^{-\ln\lL_1}\xt\lp)(B_{x,x}^{-1})=(e^{\ln\lL_1}B_{x,x})\,B_{x,x}^{-\lp}=\det B_{x,x}^{\ln/p}\,B_{x,x}^{-\lp}=\lD_{x,x}^\ln\,B_{x,x}^{-\lp}.$$
Since $\dl_0\gl_x=B_{x,x}^{1/2}$ it follows that the inner product for the little Hilbert spaces \er{63} corresponds to the choice of holomorphic representation $e^{-\ln\lL_1}\xt\lp.$ By \cite{W} there exists a cut-off $\ln_\ll$ such that the continuous part of the Wallach set is given by $\ln>\ln_\ll.$ In the special case of the unit disk $\Dl$, the little Hilbert spaces $H^2_\ln(\Dl,\PL_\ll\Cl)$ are generated by the kernel functions
$$(1-z\o w)^{-\ln}(1-z\o w)^{-2\ll}=(1-z\o w)^{-\ln-2\ll}$$
for $0\le\ll\le n.$ Thus for the unit disk the continuous part of the Wallach set is determined by $\ln+2\ll>0.$ In general, we have 
\begin{proposition}\label{h} Let $\ln>\ln_\ll+\f a2(r-1).$ Then the multiplication operators $z_i^\oc$ by the coordinate functions are bounded on the Hilbert space $H^2_\ln(D,\PL_\ll Z).$  
\end{proposition}
\begin{proof} Since $\ln=\ln_\ll+\Le+\f a2(r-1)+\Le$ for some $\Le>0$ by assumption, it follows that 
$$(b^2-(z|w))(\KL_w^{\ln,\ll}\lF)_z(\lz)=(b^2-(z|w))\lD_{z,w}^{-\ln}\lF(B^{-1}_{z,w}\lz)
=\((b^2-(z|w))\lD_{z,w}^{-a(r-1)/2-\Le}\)\lD_{z,w}^{-\ln_\ll-\Le}\lF(B^{-1}_{z,w}\lz).$$
By the definition of $\ln_\ll,$ we see that $\lD_{z,w}^{-\ln_\ll-\Le}\lF(B^{-1}_{z,w}\lz)$ is positive for any $\Le>0.$ 
It is well-known (cf. \cite[Theorem 4.1]{AZ} and \cite[Theorem 1.1]{BM}) that the multiplication operators on the weighted Bergman spaces $H^2_\ln(D,\Cl)$ are bounded for $\ln>\f a2(r-1).$ Hence there exists $b$ for which $(b^2-(z|w))\lD_{z,w}^{-a(r-1)/2-\Le}$ is positive. Since the Schur product (point-wise product) of two positive kernels is again positive, the proof is complete.
\end{proof} 
As the main result of this section we give a criterion for boundedness on the 'big' Hilbert space $H^2_\ln(D,\PL_\cl^nZ),$ generalizing the argument given in \cite[Theorem 4.1]{KM0}.
\begin{lemma} Consider the matrix $A(\ln)=(a_\lm^\ll(\ln))$ defined in \er{37}. Then for any choice $\ln,\ln_*>\f a2(r-1)$, the open set
$$\CL_{\ln_*}^\ln:=\{\cl>0:\,A(\ln_*)^{-1}\,A(\ln)\cl>0\}$$
is non-empty.
\end{lemma} 
\begin{proof} The condition
$$\cl'=A(\ln_*)^{-1}\,A(\ln)\cl>0$$
leads to a finite set of inequalities which can be solved one-by-one since the matrices $A(\ln)$ are lower-triangular. 
\end{proof}
\begin{theorem} Let $\ln_n:=max\{\ln_\ll:\,\ll\le n\}.$ Then the multiplication operators $z_1^\oc,\ldots,z_d^\oc$ are bounded on the Hilbert space $H^2_\ln(D,\PL_\cl^nZ)$ whenever  $\ln>\ln_n+\f a2(r-1)$  and $\cl\in\CL^\ln_{\ln_n}.$ For the unit ball, the multiplication operators are bounded for $\ln>\ln_n$ and every $\cl>0.$  
\end{theorem}
\begin{proof}  By \er{38} and \er{39}, we have
$$\KL^{\ln,\cl}_{z,w}=\lD_{z,w}^{-\ln}\,\JL^{\ln,\cl}_{z,w}\qquad\forall\,z,w\in D,$$
where
\be{40}\JL_{z,w}^{\ln,\cl}=B_{z,w}^{\d\lp_n/2}\,\o\tL_w^{\d\lp_n}\,\KL_{0,0}^{\ln,\cl}\,\o\tL_z^{\d\lp_n\,*}\,B_{z,w}^{\d\lp_n/2}.\ee
Now let $\ln>\ln_n+\f a2(r-1).$ Assume that $\cl>0$ and $\ln_n<\ln'<\ln-\f a2(r-1)$ satisfy
$$\cl'=A(\ln')^{-1}A(\ln)\cl>0.$$
Then $A(\ln)\cl=A(\ln')\cl'$ and Theorem \er{g} implies $\KL_{0,0}^{\ln,\cl}=\KL_{0,0}^{\ln',\cl'}$. By \er{40}, it follows that 
$\JL^{\ln,\cl}_{z,w}=\JL^{\ln',\cl'}_{z,w}.$ 
Since $\ln-\ln'>\f a2(r-1),$ it follows as in the proof of Proposition \er{h} that $(b^2-(z|w))\lD_{z,w}^{\ln'-\ln}$ is positive for some constant $b$. For each partition $\ll\le n$  we have $\ln'>\ln_\ll$ by assumption, and hence $\ln'$ belongs to the Wallach set for $\ll,$ i.e., $\KL_{z,w}^{\ln',\ll}$ is positive. Since $\cl'>0$ it follows from \er{35} that $\KL_{z,w}^{\ln',\cl'}$ is also positive.
Thus
$$(b^2-(z|w))\KL^{\ln,\cl}_{z,w}=(b^2-(z|w))\lD_{z,w}^{-\ln}\JL^{\ln,\cl}_{z,w}=\((b^2-(z|w))\lD_{z,w}^{\ln'-\ln}\)\lD_{z,w}^{-\ln'}\JL^{\ln,\cl}_{z,w}$$
$$=\((b^2-(z|w))\lD_{z,w}^{\ln'-\ln}\)\lD_{z,w}^{-\ln'}\JL^{\ln',\cl'}_{z,w}=\((b^2-(z|w))\lD_{z,w}^{\ln'-\ln}\)\KL^{\ln',\cl'}_{z,w}.$$
is again realized as the Schur product of two positive kernels. In the {\bf general case} suppose that $\ln>\ln_n+\f a2(r-1)$ and 
$\cl\in\CL^\ln_{\ln_n}.$ Then $A(\ln_n)^{-1}A(\ln)\cl>0.$ Since $A(\ln)$ depends continuously on the parameter $\ln$ there exists 
$\Le>0$ such that $\ln>\ln_n+\f a2(r-1)+\Le$ and $\cl':=A(\ln_n+\Le)^{-1}A(\ln)\cl>0.$ Putting $\ln':=\ln_n+\Le$ we have 
$A(\ln')\cl'=A(\ln)\cl$ and $\ln-\ln'>\f a2(r-1).$ Thus the conditions for the above argument hold. In the {\bf rank $1$-case}, suppose that $\ln>\ln_n$ and $\cl>0.$ By continuity, there exists $\Le>0$ such that $\ln':=\ln-\Le>\ln_n$ and $\cl':=A(\ln-\Le)^{-1}A(\ln)\cl>0.$
\end{proof}

\section{Irreducibility} 
Unlike \er{19}, the representation \er{24} is not irreducible since the 'typical fibre' $\PL^nZ$ is not irreducible under $K.$ Remarkably, as shown in this section, the 'big' Hilbert space is still irreducible as a {\bf Hilbert module}, i.e., under the multiplication action by bounded holomorphic functions and their adjoints. This result leads to the important problem to classify all irreducible representations of the Toeplitz $C^*$-algebra acting on $H^2_\ln(D,\PL^nZ),$ solved for scalar-valued Bergman spaces in \cite{U1,U2}.

Let $A:\HL\to\HL^\1$ be an intertwining operator between reproducing kernel Hilbert spaces of holomorphic maps $D\to V$, in the sense that
$$f^\oc A=A f^\oc$$
for all holomorphic multipliers $f$ on $D.$ Here $f^\oc$ denotes the module (multiplication) action.
\begin{lemma} Consider an orthonormal basis $\lf_i\in V.$ Then any holomorphic map $\lF$ in $\HL$ has a representation
$$\lF=\S_i(\lf_i|\lF)^\oc\lf_i,$$
where the multiplier is induced by the holomorphic function $z\mapsto(\lf_i|\lF_z).$ 
\end{lemma}
\begin{proof} Write $\lF$ as a sum of terms $\lF_z=f(z)\lf,$ where $f\in\OL(D)$ and $\lf\in V.$ Then
$$\lF_z=f(z)\lf=\S_i f(z)\lf_i(\lf_i|\lf)=\S_i\lf_i(\lf_i|f(z)\lf)=\S_i\lf_i(\lf_i|\lF_z).$$
\end{proof}
\begin{lemma} Define $A_z\in Hom(V,V^\1)$ by $A_z\lf:=(A\lf)_z$ for all $\lf\in V,\,z\in D.$ Then
$$(A\lF)_z=A_z\lF_z.$$
\end{lemma}
\begin{proof} The intertwining property implies
$$A\lF=A\S_i(\lf_i|\lF)^\oc\lf_i=\S_i(\lf_i|\lF)^\oc A\lf_i$$
and hence
$$(A\lF)_z=\S_i(\lf_i|\lF_z)(A\lf_i)_z=\S_i(\lf_i|\lF_z)A_z\lf_i=A_z(\S_i(\lf_i|\lF_z)\lf_i)=A_z\lF_z.$$
\end{proof}
\begin{proposition}\label{b} Any intertwining operator $A$ satisfies
$$A^*\KL_w^\1=\KL_wA_w^*$$
for all $w\in D$ and the respective kernel functions.
\end{proposition}
\begin{proof} Let $\lf\in V,\,\lq\in V^\1.$ Then
$$(\KL_z\lf|A^*\KL_w^\1\lq)=(A\KL_z\lf|\KL_w^\1\lq)=((A\KL_z\lf)_w|\lq)=(A_w(\KL_z\lf)_w|\lq)$$
$$=(A_w\KL_{w,z}\lf|\lq)=(\lf|\KL_{z,w}A_w^*\lq)=(\lf|\KL_z^*\KL_w A_w^*\lq)=(\KL_z\lf|\KL_wA_w^*\lq).$$
\end{proof}
As special cases we obtain
\begin{proposition} If $U$ is a unitary intertwiner, then
$$\KL_{z,w}^\1=U_z\KL_{z,w}U_w^*$$
where $U:D\to GL(V)$ is a holomorphic mapping. If $\HL=\HL^\1$ and $P$ is a self-adjoint intertwiner, then
\be{27}P_z\KL_{z,w}=\KL_{z,w}P_w^*.\ee
\end{proposition}
\begin{proof} Since $U^*\KL_w^\1=\KL_w U_w^*$ by Proposition \er{b} we obtain $\KL_w^\1\lq=U\KL_wU_w^*\lq$ for all $\lq\in V^\1.$
It follows that
$$\KL_{z,w}^\1\lq=(\KL_w^\1\lq)_z=(U\KL_wV_w^*\lq)_z=U_z(\KL_wU_w^*\lq)_z=U_z\KL_{z,w}U_w^*\lq.$$
In the self-adjoint case Proposition \er{b} implies $P\KL_w=\KL_wP_w^*$ and we obtain 
$$\KL_{z,w}P_w^*\lf=(\KL_wP_w^*\lf)_z=(P\KL_w\lf)_z=P_z(\KL_w\lf)_z=P_z\KL_{z,w}\lf$$
for all $\lf\in V.$
\end{proof}
\begin{lemma}\label{i} Let $P$ be a self-adjoint intertwiner. Then we have 
$$P_0^*\KL_{z,0}^{-1}\KL_{z,w}\KL_{0,w}^{-1}=\KL_{z,0}^{-1}\KL_{z,w}\KL_{0,w}^{-1}P_0$$
\end{lemma}
\begin{proof} Putting $w=0$ in \er{27} we obtain $P_z\KL_{z,0}=\KL_{z,0}P_0^*.$ Therefore $P_z=\KL_{z,0}P_0^*\KL_{z,0}^{-1}$
and hence $P_w^*=\KL_{0,w}^{-1}P_0 \KL_{0,w}.$ With \er{27} it follows that
$$\KL_{z,0}P_0^*\KL_{z,0}^{-1}\KL_{z,w}=\KL_{z,w}\KL_{0,w}^{-1}P_0 \KL_{0,w}.$$
\end{proof}

For the 'big' Hilbert space, the cocycle \er{25} yields
$$\JL_{x,x}=B_{x,x}^{\lp_n/2}\,\o\tL_x^{\lp_n}\,\KL_{0,0}^\ln\,(\o\tL_x^{\lp_n})^*\,B_{x,x}^{\lp_n/2},$$
with the adjoint depending on the choice of $K$-invariant inner product on $\PL_\cl^nZ.$ Polarizing we obtain 
\be{39}\JL_{z,w}=B_{z,w}^{\lp_n/2}\,\o\tL_w^{\lp_n}\,\KL_{0,0}^\ln\,(\o\tL_z^{\lp_n})^*\,B_{z,w}^{\lp_n/2}\qquad\forall\,z,w\in D.\ee
\begin{theorem} The $C^*$-algebra generated by polynomial multipliers, acting on the 'big' Hilbert space $H^2_\ln(D,\PL_\cl^nZ),$ is irreducible.
\end{theorem}
\begin{proof} Let $P$ be a self-adjoint projection on $H^2_\ln(D,\PL_\cl^nZ)$ commuting with polynomial multipliers. Replacing $P$ by
\be{30}\I_{K}dk\,k^{\lp_n}\,P\,k^{-\lp_n}\ee
we may assume that $P$ is $K$-invariant. Then the idempotent endomorphism $P_0$ on $\PL^nZ$ is also $K$-invariant. Since $K$ acts without multiplicity it follows that $P_0$ is unitarily equivalent to a diagonal matrix with eigenvalues $0,1.$ Thus $P_0=P_0^*.$ 
We put $L:=\KL_{0,0}^\ln$ and continue with the following two claims:
\begin{claim}\label{c} $P_0$ commutes with all operators of the form
\be{28}(\o\tL_{-z^w}^{\lp_n})^*\,L^{-1}\,\o\tL_w^{\lp_n}\,L\,(\o\tL_z^{\lp_n})^*\,L^{-1}\,\o\tL_{-w^z}^{\lp_n},\ee
where $(z,w)\in Z\xx Z$ is quasi-invertible.
\end{claim}
\begin{proof} Applying \er{29} to $h=B_{z,w}^{1/2}$ we obtain
$$B_{z,w}^{1/2}\o\tL_{-w}=\o\tL_{-B_{w,z}^{-1/2}w}\,B_{z,w}^{1/2}=\o\tL_{-w^z}\,B_{z,w}^{1/2}.$$ 
Suppose first that $z,w\in D.$ Since $P_0=P_0^*$ and $\KL_{z,w}$ and $\JL_{z,w}$ differ only by a scalar factor, Lemma \er{i} implies that $P_0$ commutes with 
$$\JL_{z,0}^{-1}\JL_{z,w}\JL_{0,w}^{-1}=(\o\tL_{-z}^{\lp_n})^*\,L^{-1}\,B_{z,w}^{\lp_n/2}\,\o\tL_w^{\lp_n}\,L\,(\o\tL_z^{\lp_n})^*
\,B_{z,w}^{\lp_n/2}\,L^{-1}\,\o\tL_{-w}^{\lp_n}$$
$$=(\o\tL_{-z}^{\lp_n})^*\,B_{z,w}^{\lp_n/2}\,L^{-1}\,\o\tL_w^{\lp_n}\,L\,(\o\tL_z^{\lp_n})^*
\,L^{-1}\,B_{z,w}^{\lp_n/2}\,\o\tL_{-w}^{\lp_n}
=(B_{w,z}^{\lp_n/2}\,\o\tL_{-z}^{\lp_n})^*\,L^{-1}\,\o\tL_w^{\lp_n}\,L\,(\o\tL_z^{\lp_n})^*
\,L^{-1}\,B_{z,w}^{\lp_n/2}\,\o\tL_{-w}^{\lp_n}$$
$$=(\o\tL_{-z^w}\,B_{w,z}^{1/2})^*\,L^{-1}\,\o\tL_w^{\lp_n}\,L\,(\o\tL_z^{\lp_n})^*
\,L^{-1}\,\o\tL_{-w^z}\,B_{z,w}^{1/2}
=B_{z,w}^{\lp_n/2}\,(\o\tL_{-z^w}^{\lp_n})^*\,L^{-1}\,\o\tL_w^{\lp_n}\,L\,(\o\tL_z^{\lp_n})^*\,L^{-1}\,\o\tL_{-w^z}^{\lp_n}\,B_{z,w}^{\lp_n/2}.$$
Since $P_0$ commutes with the endomorphisms $B_{z,w}^{\lp_n/2}\in K^\Cl,$ the assertion follows for $z,w\in D.$ Since \er{28} is sesqui-holomorphic in $(z,w),$ the assertion follows in general.
\end{proof}
The center $\Cl^\xx\ic K^\Cl$ acts on $\PL^n Z$ via
$$(s^{-\lp_n}\lq)(\lz)=s^{-nd/p}\,\lq(s\lz)\qquad\forall\,s\in\Cl^\xx.$$ 
As a special case of \er{29} we have
$$s^{-\lp_n}\o\tL_z^{\lp_n}\,s^{\lp_n}=\o\tL_{\o sz}^{\lp_n}.$$
An element $c\in Z$ is called a {\bf tripotent} if $Q_cc=c.$
\begin{claim} Let $c\in Z$ be a tripotent. Then $P_0$ commutes with all operators $(\o\tL_{tc}^{\lp_n})^*\,L^{-1}\,\o\tL_{tc}^{\lp_n},$ where $t\in\Rl.$
\end{claim}
\begin{proof} Let $r\in\Rl,\,r>1,$ and put $z=w=rc.$ Then
$$z^w=\f r{1-r^2}c=-s^2z,$$
where $s=\f1{\F{r^2-1}}.$ By Lemma \er{c} $P_0$ commutes with the operators $(\o\tL_{s^2z}^{\lp_n})^*\,L^{-1}\,\o\tL_z^{\lp_n}\,L\,(\o\tL_z^{\lp_n})^*\,L^{-1}\,\o\tL_{s^2z}^{\lp_n}$ and therefore also with the operators
$$L^{1/2}\,s^{-2\lp_n}\,(\o\tL_{s^2z}^{\lp_n})^*\,L^{-1}\,\o\tL_z^{\lp_n}\,L\,(\o\tL_z^{\lp_n})^*\,L^{-1}\,\o\tL_{s^2z}^{\lp_n}
\,s^{-2\lp_n}\,L^{1/2}=\(L^{1/2}\,s^{-\lp_n}(\o\tL_{sz}^{\lp_n})^*\,L^{-1}\,\o\tL_{sz}^{\lp_n}\,s^{-\lp_n}L^{1/2}\)^2.$$
Taking the positive definite square-root, it follows that $P_0$ commutes with $L^{1/2}\,s^{-\lp_n}(\o\tL_{sz}^{\lp_n})^*\,L^{-1}\,\o\tL_{sz}^{\lp_n}\,s^{-\lp_n}L^{1/2}$ and hence also with 
$$(\o\tL_{sz}^{\lp_n})^*\,L^{-1}\,\o\tL_{sz}^{\lp_n}=(\o\tL_{tc}^{\lp_n})^*\,L^{-1}\,\o\tL_{tc}^{\lp_n},$$ 
where
$$sz=\f{r}{\F{r^2-1}}c=tc$$
is any multiple of $c$ with $t>1.$ By analytic continuation, the assertion follows.
\end{proof}
We will now complete the proof of the Theorem. Since $\o\tL_{tc}^{\lp_n}=\exp(t\,Q_\lz c)^{\d\lp_n},$ taking the derivative at $t=0$ implies that $P_0$ commutes with
$$((Q_\lz c)^{\d\lp_n})^*L^{-1}+L^{-1}(Q_\lz c)^{\d\lp_n}.$$
According to Theorem \ref{d} $(Q_\lz c)^{\d\lp_n}$ is a strictly upper triangular operator, whereas $((Q_\lz c)^{\d\lp_n})^*$ is strictly lower triangular. Hence $P_0$ commutes with $((Q_\lz c)^{\d\lp_n})^*L^{-1}$ and $L^{-1}(Q_\lz c)^{\d\lp_n}$ separately. In particular, $P_0$ commutes with $(Q_\lz c)^{\d\lp_n}$ for any tripotent $c.$ By the spectral theorem for Jordan triples \cite{FK,L}, $P_0$ commutes with $(Q_\lz c)^{\d\lp_n}$ for all $c\in Z.$ This implies that $P_0$ commutes with the action of $K^\Cl\,P_+.$ Since this group acts transitively on $\PL^n Z$ it follows that $P_0=0$ or $P_0=1,$ for the $K$-invariant integral \er{30}. This implies that the original $P_0$ is also trivial.   
\end{proof}
Note that the above proof uses mainly the facts that the typical fibre $V=\PL^nZ$ carries an irreducible holomorphic representation of 
$K^\Cl\,P_+$ which is multiplicity-free under restriction to $K.$

\end{document}